\documentclass[12pt]{amsart}       
\usepackage{txfonts}
\usepackage{amssymb}
\usepackage{eucal}
\usepackage{graphicx}
\usepackage{amsmath}
\usepackage{amscd}
\usepackage[all]{xy}           
\usepackage{amsfonts,latexsym}
\usepackage{xspace}
\usepackage{epsfig}
\usepackage{float}
\usepackage{color}
\usepackage{fancybox}
\usepackage{colordvi}
\usepackage{multicol}
\usepackage{colordvi}
\usepackage[colorlinks,final,backref=page,hyperindex,hypertex]{hyperref}
\usepackage[active]{srcltx} 
\usepackage{tikz}

\usepackage{graphicx}
\usepackage{mathrsfs}
\newtheorem{theorem}{Theorem}[section]
\newtheorem{prop}[theorem]{Proposition}

\newtheorem{coro}[theorem]{Corollary}
\newtheorem{prop-def}{Proposition-Definition}[section]
\newtheorem{coro-def}{Corollary-Definition}[section]

\theoremstyle{definition}
\newtheorem{defn}[theorem]{Definition}
\newtheorem{remark}[theorem]{Remark}

\newtheorem{exam}[theorem]{Example}

\newcommand{\nc}{\newcommand}
\nc{\tred}[1]{\textcolor{red}{#1}}
\nc{\tblue}[1]{\textcolor{blue}{#1}}
\nc{\tgreen}[1]{\textcolor{green}{#1}}
\nc{\tpurple}[1]{\textcolor{purple}{#1}}
\nc{\btred}[1]{\textcolor{red}{\bf #1}}
\nc{\btblue}[1]{\textcolor{blue}{\bf #1}}
\nc{\btgreen}[1]{\textcolor{green}{\bf #1}}
\nc{\btpurple}[1]{\textcolor{purple}{\bf #1}}
\nc{\NN}{{\mathbb N}}
\nc{\ncsha}{{\mbox{\cyr X}^{\mathrm NC}}} \nc{\ncshao}{{\mbox{\cyr
X}^{\mathrm NC}_0}}

\newcommand{\efootnote}[1]{}

\renewcommand{\textbf}[1]{}

\newcommand{\delete}[1]{}

\nc{\mlabel}[1]{\label{#1}}  
\nc{\mcite}[1]{\cite{#1}}  
\nc{\mref}[1]{\ref{#1}}  
\nc{\mbibitem}[1]{\bibitem{#1}} 

\delete{
\nc{\mlabel}[1]{\label{#1}{\hfill \hspace{1cm}{\bf{{\ }\hfill(#1)}}}}
\nc{\mcite}[1]{\cite{#1}{{\bf{{\ }(#1)}}}}  
\nc{\mref}[1]{\ref{#1}{{\bf{{\ }(#1)}}}}  
\nc{\mbibitem}[1]{\bibitem[\bf #1]{#1}} 
}

\nc{\opa}{\ast} \nc{\opb}{\odot} \nc{\op}{\bullet} \nc{\pa}{\frakL}
\nc{\arr}{\rightarrow} \nc{\lu}[1]{(#1)} \nc{\mult}{\mrm{mult}}
\nc{\diff}{\mathfrak{Diff}}
\nc{\opc}{\sharp}\nc{\opd}{\natural}
\nc{\ope}{\circ}
\nc{\dpt}{\mathrm{d}}
\nc{\hck}{H_{RT}}
\nc{\vdf}{\calf}
\nc{\ldf}{\calf_\ell}
\nc{\hlf}{H_\ell}
\nc{\onek}{\mathbf{1}_\bfk}

\nc{\tforall}{\quad \text{ for all }}

\nc{\mrba}{matching  Rota-Baxter algebra\xspace}
\nc{\Mrba}{Matching multiple Rota-Baxter algebra\xspace}
\nc{\mrbas}{matching multiple Rota-Baxter algebras\xspace}
\nc{\Mrbas}{Matching multiple Rota-Baxter algebras\xspace}

\nc{\matmul}{matching multiple\xspace}
\nc{\Matmul}{Matching multiple\xspace}
\nc{\match}{matching\xspace}
\nc{\Match}{Matching\xspace}

\nc{\paybe}{polarized associative Yang-Baxter equation\xspace}
\nc{\Paybe}{Polarized associative Yang-Baxter equation\xspace}
\nc{\cpaybe}{PAYBE}

\nc{\diam}{alternating\xspace}
\nc{\Diam}{Alternating\xspace}
\nc{\cdiam}{canonical alternating\xspace}
\nc{\Cdiam}{Canonical alternating\xspace}
\nc{\AW}{\mathcal{A}}

\nc{\rba}{Rota-Baxter algebra\xspace}
\nc{\rbas}{Rota-Baxter algebras\xspace}

\nc{\ari}{\mathrm{ar}}

\nc{\lef}{\mathrm{lef}}

\nc{\Sh}{\mathrm{ST}}

\nc{\Cr}{\mathrm{Cr}}

\nc{\st}{{Schr\"oder tree}\xspace}
\nc{\sts}{{Schr\"oder trees}\xspace}

\nc{\vertset}{\Omega} 

\nc{\assop}{\quad \begin{picture}(5,5)(0,0)
\line(-1,1){10}
\put(-2.2,-2.2){$\bullet$}
\line(0,-1){10}\line(1,1){10}
\end{picture} \quad \smallskip}

\nc{\operator}{\begin{picture}(5,5)(0,0)
\line(0,-1){6}
\put(-2.6,-1.8){$\bullet$}
\line(0,1){9}
\end{picture}}

\nc{\idx}{\begin{picture}(6,6)(-3,-3)
\put(0,0){\line(0,1){6}}
\put(0,0){\line(0,-1){6}}
\end{picture}}

\nc{\pb}{{\mathrm{pb}}}
\nc{\Lf}{{\mathrm{Lf}}}

\nc{\lft}{{left tree}\xspace}
\nc{\lfts}{{left trees}\xspace}

\nc{\fat}{{fundamental averaging tree}\xspace}

\nc{\fats}{{fundamental averaging trees}\xspace}
\nc{\avt}{\mathrm{Avt}}

\nc{\rass}{{\mathit{RAss}}}

\nc{\aass}{{\mathit{AAss}}}

\nc{\vin}{{\mathrm Vin}}    
\nc{\lin}{{\mathrm Lin}}    
\nc{\inv}{\mathrm{I}n}
\nc{\gensp}{V} 
\nc{\genbas}{\mathcal{V}} 
\nc{\bvp}{V_P}     
\nc{\gop}{{\,\omega\,}}     

\nc{\bin}[2]{ (_{\stackrel{\scs{#1}}{\scs{#2}}})}  
\nc{\binc}[2]{ \left (\!\! \begin{array}{c} \scs{#1}\\
    \scs{#2} \end{array}\!\! \right )}  
\nc{\bincc}[2]{  \left ( {\scs{#1} \atop
    \vspace{-1cm}\scs{#2}} \right )}  
\nc{\bs}{\bar{S}} \nc{\cosum}{\sqsubset} \nc{\la}{\longrightarrow}
\nc{\rar}{\rightarrow} \nc{\dar}{\downarrow} \nc{\dprod}{**}
\nc{\dap}[1]{\downarrow \rlap{$\scriptstyle{#1}$}}
\nc{\md}{\mathrm{dth}} \nc{\uap}[1]{\uparrow
\rlap{$\scriptstyle{#1}$}} \nc{\defeq}{\stackrel{\rm def}{=}}
\nc{\disp}[1]{\displaystyle{#1}} \nc{\dotcup}{\
\displaystyle{\bigcup^\bullet}\ } \nc{\gzeta}{\bar{\zeta}}
\nc{\hcm}{\ \hat{,}\ } \nc{\hts}{\hat{\otimes}}
\nc{\barot}{{\otimes}} \nc{\free}[1]{\bar{#1}}
\nc{\uni}[1]{\tilde{#1}} \nc{\hcirc}{\hat{\circ}} \nc{\lleft}{[}
\nc{\lright}{]} \nc{\lc}{\lfloor} \nc{\rc}{\rfloor}
\nc{\curlyl}{\left \{ \begin{array}{c} {} \\ {} \end{array}
    \right .  \!\!\!\!\!\!\!}
\nc{\curlyr}{ \!\!\!\!\!\!\!
    \left . \begin{array}{c} {} \\ {} \end{array}
    \right \} }
\nc{\longmid}{\left | \begin{array}{c} {} \\ {} \end{array}
    \right . \!\!\!\!\!\!\!}
\nc{\onetree}{\bullet} \nc{\ora}[1]{\stackrel{#1}{\rar}}
\nc{\ola}[1]{\stackrel{#1}{\la}}
\nc{\ot}{\otimes} \nc{\mot}{{{\boxtimes\,}}}
\nc{\otm}{\overline{\boxtimes}} \nc{\sprod}{\bullet}
\nc{\scs}[1]{\scriptstyle{#1}} \nc{\mrm}[1]{{\rm #1}}
\nc{\margin}[1]{\marginpar{\rm #1}}   
\nc{\dirlim}{\displaystyle{\lim_{\longrightarrow}}\,}
\nc{\invlim}{\displaystyle{\lim_{\longleftarrow}}\,}
\nc{\mvp}{\vspace{0.3cm}} \nc{\tk}{^{(k)}} \nc{\tp}{^\prime}
\nc{\ttp}{^{\prime\prime}} \nc{\svp}{\vspace{2cm}}
\nc{\vp}{\vspace{8cm}} \nc{\proofbegin}{\noindent{\bf Proof: }}
\nc{\proofend}{$\blacksquare$ \vspace{0.3cm}}
\nc{\modg}[1]{\!<\!\!{#1}\!\!>}
\nc{\intg}[1]{F_C(#1)} \nc{\lmodg}{\!
<\!\!} \nc{\rmodg}{\!\!>\!}
\nc{\cpi}{\widehat{\Pi}}
\nc{\sha}{{\mbox{\cyr X}}}  

\newfont{\scyr}{wncyr10 scaled 550}
\nc{\ssha}{\mbox{\bf \scyr X}}
\nc{\shap}{{\mbox{\cyrs X}}} 
\nc{\shpr}{\diamond}    
\nc{\shp}{\ast} \nc{\shplus}{\shpr^+}
\nc{\shprc}{\shpr_c}    
\nc{\msh}{\ast} \nc{\zprod}{m_0} \nc{\oprod}{m_1}
\nc{\vep}{\epsilon} \nc{\labs}{\mid\!} \nc{\rabs}{\!\mid}
\nc{\sqmon}[1]{\langle #1\rangle}

\nc{\mmbox}[1]{\mbox{\ #1\ }} \nc{\dep}{\mrm{dep}} \nc{\fp}{\mrm{FP}}
\nc{\rchar}{\mrm{char}} \nc{\End}{\mrm{End}} \nc{\Fil}{\mrm{Fil}}
\nc{\Mor}{Mor\xspace} \nc{\gmzvs}{gMZV\xspace}
\nc{\gmzv}{gMZV\xspace} \nc{\mzv}{MZV\xspace}
\nc{\mzvs}{MZVs\xspace} \nc{\Hom}{\mrm{Hom}} \nc{\id}{\mrm{id}}
\nc{\im}{\mrm{im}} \nc{\incl}{\mrm{incl}} \nc{\map}{\mrm{Map}}
\nc{\mchar}{\rm char} \nc{\nz}{\rm NZ} \nc{\supp}{\mathrm Supp}

\nc{\Alg}{\mathbf{Alg}} \nc{\Bax}{\mathbf{Bax}} \nc{\bff}{\mathbf f}
\nc{\bfk}{{\bf k}} \nc{\bfone}{{\bf 1}} \nc{\bfx}{\mathbf x}
\nc{\bfy}{\mathbf y}
\nc{\base}[1]{\bfone^{\otimes ({#1}+1)}} 
\nc{\Cat}{\mathbf{Cat}}

\nc{\detail}{\marginpar{\bf More detail}
    \noindent{\bf Need more detail!}
    \svp}
\nc{\Int}{\mathbf{Int}} \nc{\Mon}{\mathbf{Mon}}
\nc{\rbtm}{{shuffle }} \nc{\rbto}{{Rota-Baxter }}
\nc{\remarks}{\noindent{\bf Remarks: }} \nc{\Rings}{\mathbf{Rings}}
\nc{\Sets}{\mathbf{Sets}} \nc{\wtot}{\widetilde{\odot}}
\nc{\wast}{\widetilde{\ast}} \nc{\bodot}{\bar{\odot}}
\nc{\bast}{\bar{\ast}} \nc{\hodot}[1]{\odot^{#1}}
\nc{\hast}[1]{\ast^{#1}} \nc{\mal}{\mathcal{O}}
\nc{\tet}{\tilde{\ast}} \nc{\teot}{\tilde{\odot}}
\nc{\oex}{\overline{x}} \nc{\oey}{\overline{y}}
\nc{\oez}{\overline{z}} \nc{\oef}{\overline{f}}
\nc{\oea}{\overline{a}} \nc{\oeb}{\overline{b}}
\nc{\weast}[1]{\widetilde{\ast}^{#1}}
\nc{\weodot}[1]{\widetilde{\odot}^{#1}} \nc{\hstar}[1]{\star^{#1}}
\nc{\lae}{\langle} \nc{\rae}{\rangle}
\nc{\lf}{\lfloor}
\nc{\rf}{\rfloor}

\nc{\QQ}{{\mathbb Q}}
\nc{\RR}{{\mathbb R}} \nc{\ZZ}{{\mathbb Z}}

\nc{\cala}{{\mathcal A}} \nc{\calb}{{\mathcal B}}
\nc{\calc}{{\mathcal C}}
\nc{\cald}{{\mathcal D}} \nc{\cale}{{\mathcal E}}
\nc{\calf}{{\mathcal F}} \nc{\calg}{{\mathcal G}}
\nc{\calh}{{\mathcal H}} \nc{\cali}{{\mathcal I}}
\nc{\call}{{\mathcal L}} \nc{\calm}{{\mathcal M}}
\nc{\caln}{{\mathcal N}} \nc{\calo}{{\mathcal O}}
\nc{\calp}{{\mathcal P}} \nc{\calr}{{\mathcal R}}
\nc{\cals}{{\mathcal S}} \nc{\calt}{{\mathcal T}}
\nc{\calu}{{\mathcal U}} \nc{\calw}{{\mathcal W}} \nc{\calk}{{\mathcal K}}
\nc{\calx}{{\mathcal X}} \nc{\CA}{\mathcal{A}}

\nc{\fraka}{{\mathfrak a}} \nc{\frakA}{{\mathfrak A}}
\nc{\frakb}{{\mathfrak b}} \nc{\frakB}{{\mathfrak B}}
\nc{\frakc}{{\mathfrak c}}
\nc{\frakD}{{\mathfrak D}} \nc{\frakF}{\mathfrak{F}}
\nc{\frakf}{{\mathfrak f}} \nc{\frakg}{{\mathfrak g}}
\nc{\frakH}{{\mathfrak H}} \nc{\frakL}{{\mathfrak L}}
\nc{\frakM}{{\mathfrak M}} \nc{\bfrakM}{\overline{\frakM}}
\nc{\frakm}{{\mathfrak m}} \nc{\frakP}{{\mathfrak P}}
\nc{\frakN}{{\mathfrak N}} \nc{\frakp}{{\mathfrak p}}
\nc{\frakS}{{\mathfrak S}} \nc{\frakT}{\mathfrak{T}}
\nc{\frakX}{{\mathfrak X}}
\nc{\BS}{\mathbb{S
}}

\font\cyr=wncyr10 \font\cyrs=wncyr7
\nc{\li}[1]{\textcolor{red}{#1}}
\nc{\lir}[1]{\textcolor{red}{Li:#1}}
\nc{\yi}[1]{\textcolor{blue}{Yi: #1}}
\nc{\xing}[1]{\textcolor{purple}{Xing:#1}}
\nc{\revise}[1]{\textcolor{purple}{#1}}

\nc{\ID}{{\rm I}}\nc{\lbar}[1]{\overline{#1}}\nc{\bre}{{\rm bre}}
\nc{\sd}{\cals}\nc{\rb}{\rm RB}\nc{\A}{\rm A}\nc{\LL}{\rm L}\nc{\tx}{\tilde{X}}
\nc{\col}{\Delta_{RT}}\nc{\mul}{m_{RT}}\nc{\ul}{u_{RT}}\nc{\epl}{\varepsilon_{RT}}
\nc{\hl}{H_{RT}}\nc{\arro}[1]{#1}\nc{\px}{P_{\tx}}\nc{\pw}{P_{\mathfrak{w}}}\nc{\pl}{B_\omega^+}
\nc{\pp}{\pl}\nc{\ppp}[1]{B^+(#1)}\nc{\dw}{\diamond_{\mathfrak{w}}}\nc{\dl}{\diamond_{\rm \ell}}
\nc{\ncshaw}{\sha^{{\rm NC}}_{\Omega}}\nc{\ncshal}{\sha^{{\rm NC}}_{{\rm RT}}}
\nc{\ver}{\rm V}\nc{\ld}{l}\nc{\del}{\Delta_{{\rm \ell}}}\nc{\epsl}{\epsilon_{{\rm \ell}}}
\nc{\uul}{u_{{\rm \ell}}}\nc{\oneh}{\mathbf{1}}\nc{\onew}{\mathbf{1}}
\nc{\etree}{1} \nc{\conc}{m_{RT}}
\nc{\hrtb}{H_{RT}(X\sqcup\Omega)} \nc{\hrts}{H_{RT}(X, \Omega)}\nc{\rts}{\mathcal{T}(X, \Omega)}\nc{\rfs}{\mathcal{F}(X, \Omega)} \nc{\ncshall}{\sha^{{\rm NC}}_{{\rm RT}}} \nc{\ldl}{\leq_{\mathrm{dl}}} \nc{\pla}{B_{\alpha}^{+}} \nc{\plb}{B_{\beta}^{+}}

\nc{\bim}[1]{#1}  \nc{\shaop}{\sha_{\Omega}^{+}}  \nc{\shao}{\sha_{\Omega}}
\nc{\bbim}[2]{#1 #2} \nc{\bbbim}[2]{#1,\, #2} \nc{\RBF}{{\rm RBF}}
\nc{\frbf}{F_{\RBF}} \nc{\shaf}{\ssha_{\tiny{\Omega}}} \nc{\sham}{\diamond_{\Omega}}

\nc{\dnx}{\Delta_n A} \nc{\dx}{\Delta A} \nc{\dgp}{{\rm deg_{P}}}
\nc{\dgt}{{\rm deg_{T}}} \nc{\dg}{{\rm deg}} \nc{\ida}{ID($A$)} \nc{\tu}{\tilde{u}} \nc{\tv}{\tilde{v}}
 \nc{\fbase}{\calb} \nc{\LF}{\mathrm{RF}} \nc{\FFA}{\mathrm{LF}} \nc{\irr}{\mathrm{Irr}}
 \nc{\result}{\bfk\mathrm{Irr}(S_n)}  \nc{\I}{I_{\mathrm{ID},n}^0}
 \nc{\nrs}{\calr_n^\star} \nc{\ii}{\mathrm{I}} \nc{\iii}{\mathrm{II}}
\nc{\intl}{{\rm int}}\nc{\ws}[1]{{#1}}\nc{\deleted}[1]{\delete{#1}}\nc{\plas}{placements\xspace}

\nc{\Id}{\mathrm{Id}} \nc{\Irr}{\mathrm{Irr}}

\nc{\tos}{totally ordered set } \nc{\nes}{nonempty set}
\nc{\baa}{\succ_\alpha}
\nc{\laa}{\prec_\alpha}
\nc{\bba}{\succ_\beta}
\nc{\lba}{\prec_\beta}
\nc{\bdt} {\succ_\delta}
\nc{\lia} {\prec_\iota}
\nc{\bka}{\succ_\kappa}
\nc{\lmu}{\prec_\mu}

\nc{\ya}{\YY{\node[scale=0.8] at (0.75,-0.55) {{$a$}};}}
\nc{\yb}{\YY{\node[scale=0.8] at (0.75,-0.55) {{$b$}};}}
\nc{\yc}{\YY{\node[scale=0.8] at (0.75,-0.55) {{$c$}};}}

\definecolor{red}{rgb}{1.,0.,0.}
\definecolor{green}{rgb}{0.,1.,0.}
\definecolor{blue}{rgb}{0.,0.,1.}

\newcommand{\tdunc}[1]{\begin{tikzpicture}[line cap=round,line join=round,
x=0.8cm,y=0.8cm]
\clip(1.8,1.9) rectangle (2.5,2.5);
\begin{scriptsize}
\draw [fill=black] (2.,2.) circle (1pt);
\end{scriptsize}
\draw(2.3,2.1) node {\tiny #1};
\end{tikzpicture}}

\newcommand{\tdtroisunc}[5]{\begin{tikzpicture}[line cap=round,line join=round,
x=0.8cm,y=0.8cm]
\clip(1.5,1.9) rectangle (2.5,3.2);
\draw [line width=0.5pt, color=#4] (2.,2.)-- (1.7,2.5);
\draw [line width=0.5pt, color=#5] (2.,2.)-- (2.3,2.5);
\begin{scriptsize}
\draw [fill=black] (1.7,2.5) circle (1pt);
\draw [fill=black] (2.,2.) circle (1pt);
\draw [fill=black] (2.3,2.5) circle (1pt);
\end{scriptsize}
\draw(2.4,2.1) node {\tiny #1};
\draw(2.3,2.8) node {\tiny #2};
\draw(1.7,2.8) node {\tiny #3};
\end{tikzpicture}}

\newcommand{\tdtroisdeuxc}[5]{\begin{tikzpicture}[line cap=round,line join=round,
x=0.8cm,y=0.8cm]
\clip(1.8,1.9) rectangle (2.5,3.5);
\draw [line width=0.5pt, color=#4] (2.,2.5)-- (2.,2.);
\draw [line width=0.5pt, color=#5] (2.,2.5)-- (2.,3.);
\begin{scriptsize}
\draw [fill=black] (2.,2.) circle (1pt);
\draw [fill=black] (2.,2.5) circle (1pt);
\draw [fill=black] (2.,3.) circle (1pt);
\end{scriptsize}
\draw(2.3,2.1) node {\tiny #1};
\draw(2.3,2.6) node {\tiny #2};
\draw(2.3,3.1) node {\tiny #3};
\end{tikzpicture}}

\newcommand{\tddeuxc}[3]{\begin{tikzpicture}[line cap=round,line join=round,
x=0.8cm,y=0.8cm]
\clip(1.8,1.9) rectangle (2.5,3.);
\draw [line width=.5pt, color=#3] (2.,2.5)-- (2.,2.);
\begin{scriptsize}
\draw [fill=black] (2.,2.) circle (1pt);
\draw [fill=black] (2.,2.5) circle (1pt);
\end{scriptsize}
\draw(2.3,2.1) node {\tiny #1};
\draw(2.3,2.6) node {\tiny #2};
\end{tikzpicture}}

\newcommand\YY[2][]{%
\tikz[line width=0.18ex,scale=0.75,baseline=-3ex,inner sep=1pt,#1]{
\draw (0,0)--+(1/2,-1/2)--+(1,0) (1/2,-1/2)--+(0,-1/2); #2}}

\newcommand\YYY[2][]{%
\tikz[line width=0.18ex,scale=0.5,baseline=-0.58ex,
every node/.style={font=\small,inner sep=1pt},#1]{
\coordinate (o) at (0,0); #2}}

\nc{\mscr}[1]{\mathscr{#1}}
\nc{\cal}[1]{\mathcal{#1}} \nc{\bb}[1]{\mathbb{#1}}
 \nc{\inn}{\rm{in}}
\nc{\spp}{\mathscr{P}} \nc{\stt}{\mscr{T}} \nc{\spb}{\mscr{P}_\circ}  \nc{\sph}{\mscr{P}_\bullet}
\nc{\dera}{\mathscr{D}er\mathscr{A}}

\nc{\add}{\oplus}
\nc{\prl}[2]{(x#1y)#2z}\nc{\prr}[2]{x#1(y#2z)}
\nc{\ra}[3]{#1(x)#2#3(y)}
\nc{\dereq}[6]
{\treey{\cdlr[0.8]{o} \cdl{ol}\cdr{or}#1{o/b}
\node at (0.2,-0.2) {$d$};\cdlr{o}\node  at (0,0) {$#2$};}
-\treey{\cdlr[0.8]{o} \cdl{ol}\cdr{or}
\node at (ol) {$#4$};\node at (-0.2,0.5) {$d$};\cdlr{o}\node  at (0,0) {$#3$};}
-\treey{\cdlr[0.8]{o} \cdl{ol}\cdr{or}
\node at (or) {$#6$};\node at (0.2,0.5) {$d$};\cdlr{o}\node  at (0,0) {$#5$};}}
\nc{\kc}[4]{\beta_{#1,#2,#3}^{b,#4}}\nc{\ka}[3]{\gamma_{#1,#2}^{c,#3}}
\nc{\otimesh}{\mathop{\otimes}_{\text{H}}}
\nc{\hu}[1]{\tgreen{\underline{Huhu:}#1 }}
\nc{\kong}{\noindent}

\usetikzlibrary{calc}
\newlength\xch\newlength\dbj
\newif\ifqdd
\newcommand\cddf[3]{%
\coordinate (#2) at ($(#1)+(#3)$);
\draw (#1)--(#2);
\ifqdd\fill (#1) circle (\dbj);\fi}
\newcommand\cdx[4][1]{\cddf{#2}{#3}{#4:#1*\xch}}
\newcommand\cdl[2][1]{\cdx[#1]{#2}{#2l}{135}}
\newcommand\cdr[2][1]{\cdx[#1]{#2}{#2r}{45}}
\newcommand\cdlr[2][1]{%
\foreach \i in {#2} {\cdl[#1]{\i}\cdr[#1]{\i}}}
\newcommand\cdb[2][1]{\cdx[#1]{#2}{#2b}{-90}}
\newcommand\treeo[2][]{\tikz[baseline=-0.58ex,line width=0.15ex,
every node/.style={font=\scriptsize,inner sep=1pt},#1]{%
\coordinate (o) at (0,0);#2}}
\newcommand\treey[2][]{\treeo[#1]{\cdb o\cdlr o#2}}
\begin{document}

\title[Matching Rota-Baxter algebras, dendriform algebras and pre-Lie algebras]{Matching Rota-Baxter algebras, matching dendriform algebras and matching pre-Lie algebras}
%
\author{Xing Gao}
\address{School of Mathematics and Statistics, Key Laboratory of Applied Mathematics and Complex Systems, Lanzhou University, Lanzhou, Gansu 730000, P.\,R. China}
\email{gaoxing@lzu.edu.cn}

\author{Li Guo}
\address{Department of Mathematics and Computer Science, Rutgers University, Newark, NJ 07102, USA}
\email{liguo@rutgers.edu}

\author{Yi Zhang} \address{School of Mathematics and Statistics,
Lanzhou University, Lanzhou, Gansu 730000, P.\,R. China}
\email{zhangy2016@lzu.edu.cn}

\date{\today}
\begin{abstract}
We introduce the notion of a \match \rba motivated by the recent work on multiple pre-Lie algebras arising from the study of algebraic renormalization of regularity structures~\mcite{BHZ,Foi18}. This notion is also related to iterated integrals with multiple kernels and solutions of the associative polarized Yang-Baxter equation.
Generalizing the natural connection of Rota-Baxter algebras with dendriform algebras to \match \rbas, we obtain the notion of \match dendriform algebras. As in the classical case of one operation, \match Rota-Baxter algebras and \match dendriform algebras are related to \match pre-Lie algebras which coincide with the aforementioned multiple pre-Lie algebras. More general notions and results on \match tridendriform algebras and \match PostLie algebras are also obtained.
\end{abstract}

\subjclass[2010]{
16W99, 
16T25, 
05C05, 
17D99 
}

\keywords{\Match Rota-Baxter algebra; polarized Yang-Baxter equation, \match dendriform algebra; \match pre-Lie algebra, \match PostLie algebra}

\maketitle

\tableofcontents

\vspace{-1cm}

\setcounter{section}{0}

\allowdisplaybreaks

\section{Introduction}
The purpose of this paper is to introduce the various concepts in order to make connection with the critical notion of multiple pre-Lie algebras that was introduced recently from the remarkable work of Bruned, Hairer and Zambotti~\mcite{BHZ} on algebraic renormalization of regularity structures.

\subsection{Structures with compatible multiple operations}
In recent years, there have been quite a few algebraic structures with multiple copies of certain given operation and with various compatibility conditions among these copies. Well-known examples include the associative dialgebra~\mcite{Lo00}, associative trialgebra~\mcite{LR04}, linearly compatible dialgebra and totally compatible dialgebra~\mcite{ZBG2}, all derived from the associative algebra, together with their variations for Lie algebra and other algebras, such as compatible Lie bialgebras~\mcite{WB15} and compatible $\mathcal{O}$-operators~\mcite{LBS}. On the level of operads, they are also studied as replicators and other compatible structures~\mcite{GK13, KV13, PBG, PBGN, Str, St2}.

Of particular interest to us is the notion of a multiple pre-Lie algebra of Foissy~\mcite{Foi18} to understand the recent work~\mcite{BHZ} on the Hopf algebra approach of quantum field theory. We note that this multiple structure is similar to another structure as mentioned above, called matching dialgebras~\mcite{ZBG1,Zin}.  The purpose of this paper is to study multiple pre-Lie algebras from this viewpoint. We first provide some backgrounds in the one operation case before extending this viewpoint to multiple operations.

The notion of pre-Lie algebras can be tracked back to the work of Vinberg~\mcite{Vin63} under the name left-symmetric algebras on convex homogeneous cones and also appeared independently  in the study of  cohomology of associative algebras~\mcite{Ger63}. It plays an important role in  many areas of mathematics and mathematical physics, such as classical and quantum Yang-Baxter equations, quantum field theory, Rota-Baxter algebras, vertex algebras and operads (see~\mcite{Bai} and references therein).

It has been found recently that in parallel to the close connection between Lie algebra and associative algebra by taking the commutator bracket, there is a close connection between the dendriform algebra and pre-Lie algebra. Moreover, in parallel to regarding the dendriform algebra as a splitting of the associative algebra, the pre-Lie algebra is a splitting of the Lie algebra in a precise sense that applies to all operads~\mcite{BBGN,PBGN}.
Additionally, the splitting of operations can be achieved by the application of the Rota-Baxter operator. So a Rota-Baxter operator on an associative algebra (resp. Lie algebra) gives rise to a dendriform algebra~\mcite{Agu01} (resp. pre-Lie algebra~\mcite{AB08}).

Thus we have the following commutative diagram of categories (the arrows will go in the opposite direction for the corresponding operads)
\begin{equation}
\begin{split}
\xymatrix{ \text{Rota-Baxter} \atop \text{associative algebra} \ar[rr]^{\text{commutator}} \ar[d] && \text{Rota-Baxter} \atop \text{Lie algebra} \ar[d] \\
\text{dendriform algebra} \ar[rr]^{\text{commutator}} && \text{pre-Lie algebra}
}
\end{split}
\mlabel{di:old}
\end{equation}
Here the horizontal arrows are taking commutator brackets and the vertical arrows are splitting of operations or applying Rota-Baxter operators.

The purpose of this paper from this viewpoint is to make these connections work for the multiple pre-Lie algebra as follows.
\begin{equation}
\begin{split}
\xymatrix{ \text{\match Rota-Baxter} \atop \text{associative algebra} \ar[rr]^{\text{commutator}} \ar[d] &&  \text{\match Rota-Baxter} \atop \text{Lie algebra} \ar[d] \\
\text{\match}\atop \text{dendriform algebra} \ar[rr]^{\text{commutator}} && \text{\match (multiple)}\atop \text{pre-Lie algebra}
}
\end{split}
\mlabel{di:new}
\end{equation}
In fact, there are variations of the above diagram in terms of linear compatibility and total compatibility. We will focus on the matching case with the application to multiple pre-Lie algebras in mind.

\subsection{Rota-Baxter algebra going multiple}

The study of the Rota-Baxter algebra originated from the probability study of G. Baxter~\cite{Bax60} and attracted quite much attention for its broad applications~\cite{CK,Gub,Ro}.

A {\bf Rota-Baxter algebra} is an associative algebra equipped with a linear operator that satisfies  the integration-by-parts formula in analysis. More precisely, for a given commutative ring $\bfk$ and $\lambda \in \bfk$, a Rota-Baxter algebra of weight $\lambda$, is a pair $(R,P)$ consisting of an associative $\bfk$-algebra $R$ and a linear operator $P:R\rightarrow R$ which satisfies the {\bf Rota-Baxter equation}
\begin{align}
P(x)P(y)=P(xP(y)+P(x)y+\lambda xy) \, \text{ for all }\, x, y\in R. \mlabel{eq:RBidd}
\end{align}
Then $P$ is called a {\bf Rota-Baxter operator} of weight $\lambda$.

One importance of the Rota-Baxter algebra is its close relationship with other algebraic structures. For example, a Rota-Baxter associative algebra naturally carries a dendriform or tridendriform algebra structure; while a Rota-Baxter Lie algebra naturally carries a pre-Lie algebra or a PostLie algebra structure. In general, an action of the Rota-Baxter operator on an operad gives a splitting of the operad~\mcite{BBGN,PBG}.

Further, it was established a long time ago that Rota-Baxter operator on a Lie algebra is precisely the operator form of the classical Yang-Baxter equation~\mcite{Bai2,STS}. More recently, similar relations between Rota-Baxter operators on an associative algebra and associative analogous of the classical Yang-Baxter equation were established~\mcite{Agu01,BGN1}.

Several recent developments have called for studying systems with multiple operators that satisfy various Rota-Baxter like compatibility conditions among the operators. The notion of a Rota-Baxter family was introduced in \mcite{EGP07} in the study of renormalization of quantum field theory. The notion of Rota-Baxter systems was introduced in~\mcite{Brz16} to gain better understanding of dendriform algebras and associative Yang-Baxter pairs~\mcite{Agu01}.

In this paper we initiate the study of another structure with multiple Rota-Baxter operators, called the \mrba, with several motivations and applications.

The first motivation is from the aforementioned notion of multiple pre-Lie algebras~\mcite{BHZ,Foi18}. There Foissy also constructed an isomorphism between  the classical operad and the the operad of multiple pre-Lie algebras on typed trees, which generalize a major result of Chapoton and Livernet~\mcite{CL01}.
Since pre-Lie algebras come naturally from Rota-Baxter operators (of weight zero) on Lie algebras, it is natural to ask, from what kind of operators that the multiple pre-Lie algebras come from, or to put it in the associative algebra context, from what kind of operators that the (yet to be defined) \match dendriform algebras come from.

The second motivation is similar to that of \mcite{Brz16} on associative Yang-Baxter equations, from another generalization of the associative Yang-Baxter equation which we call \paybe (\cpaybe). Since solutions of the associative Yang-Baxter equation naturally give Rota-Baxter operators, it is natural to determine the operator solutions of \cpaybe.

Incidently, the notion of \mrba also resolves two other issues.

The first one is the lack of structures on the set of Rota-Baxter operators. One challenge in the study of Rota-Baxter algebras is that the set of Rota-Baxter operators on a given algebra is not closed under addition, in contrast to that of differential operators. A natural question is to determine a suitable condition to be imposed to a set of Rota-Baxter operators so that linear combinations of these operators are still Rota-Baxter operators.

For the second one, studying Volterra integral operators with multiple kernels naturally calls for a family of Rota-Baxter operators (of weight zero) with suitable compatible conditions~\mcite{GGL, Ze}.

It is remarkable that the notion of a \mrba addresses all these issues. See Theorem~\mref{thm:rbmpre1}, Proposition~\mref{prop:YBandRB}, Proposition~\mref{lemm:twoopes} and Example~\mref{ex:int} respectively for the specific results.

In fact, the \match Rota-Baxter algebra of weight zero is sufficient to address most of these questions. So an extra benefit of our Rota-Baxter algebra approach is that, by considering \match Rota-Baxter operators {\em with any weights}, we are also led to other useful notions of \paybe with weights, \match tridendriform algebras generalizing the tridendriform algebra with weight~\mcite{BR10}, and \match PostLie algebras.

\subsection{Outline of the paper}

This paper is organized as follows. In Section~\mref{sec:rba}, we introduce the notion of a \match Rota-Baxter algebra and provide its examples from integrals with kernels and the polarized associative Yang-Baxter equation, as well as verifying its linearity properties.

In Section~\mref{sec:dend}, we introduce the notions of a \match dendriform algebra and a \match tridendriform algebra, and establish the connection from \match Rota-Baxter algebras to \match dendriform and tridendriform algebras. We also construct a \match dendriform algebra by typed (that is, edge decorated) planar binary trees, inspired by typed tree construction of multiple pre-Lie algebras in~\mcite{BHZ,Foi18}. Also the \match dendriform algebra is shown to be the splitting of the (linearly) compatible associative algebra.

In Section~\mref{sec:prelie}, the multiple pre-Lie algebra, identified as the \match pre-Lie algebra is shown to come from the antisymmetrization of the \match dendriform algebra and from a \match Rota-Baxter Lie algebra of weight zero. Finding a variation of \match tridendriform algebra in the Lie algebra context is not so straightforward as the natural choice does not satisfy the desired property. The right candidate is found to be the \match associative PostLie algebra, an associative variation of the \match PostLie algebra.

\smallskip

\noindent
{\bf Notation.}
Throughout this paper, let $\bfk$ be a unitary commutative ring unless the contrary is specified. It will be the base ring of all modules, algebras,  tensor products, as well as linear maps.
By an algebra we mean an associative \bfk-algebra (possibly without unit) unless otherwise specified.

\section{Matching \rbas and the \paybe}
\mlabel{sec:rba}
In this section, we first introduce the concept of \match \rbas and give their examples in various contexts. We then show that \match Rota-Baxter algebras provide the right context for the linear structure of Rota-Baxter operators. Finally, we propose the notation of the polarized associative Yang-Baxter equation and  derive a \match Rota-Baxter operator of weight $\lambda$ from a polarized associative Yang-Baxter equation of weight $\lambda$.

\subsection{\Match \rbas}
We begin with a general notion of \match \rbas for its potential applications even though much of the study later in in this paper focuses on special cases.

\begin{defn}
Let $\Omega$ be a nonempty set and let $\lambda_\Omega:=\{\lambda_\omega\,|\,\omega\in \Omega\}\subseteq \bfk$ be a set of scalars parameterized by $\Omega$. A {\bf \match (multiple) \rba} of weight $\lambda_\Omega$, or simple a {\bf \match RBA}, is a pair $(R, P_\Omega)$ consisting of an (associative) algebra $R$ and a set $P_\Omega:=\{P_\omega\,|\,\omega\in \Omega\}$ of linear operators
$$P_\omega: R\longrightarrow R, \quad \omega\in \Omega\,,$$
that satisfy the {\bf \match Rota-Baxter equation}
\begin{align}
P_\alpha(x)P_{\beta}(y)&=P_{\alpha}(xP_{\beta}(y)) +P_{\beta}(P_{\alpha}(x)y)+\lambda_\beta P_{\alpha}(xy)
\, \tforall x,y \in R, \alpha,\beta\in \Omega\,.
\mlabel{eq:RBid}
\end{align}
For \match \rbas $(R,\, P_\Omega)$ and $(R',\,P'_\Omega)$ of the same weight $\lambda_\Omega$, a {\bf \match \rba homomorphism} between them is a linear map $\phi : R\rightarrow R'$ such that $\phi P_\omega = P'_\omega \phi$ for all $\omega\in \Omega$.
When $\lambda_\Omega=\{\lambda\}$, we also call the \match \rba to have weight $\lambda$.
\mlabel{de:gmrba}
\end{defn}

\begin{remark}
\begin{enumerate}
\item
Any Rota-Baxter algebra of weight $\lambda$ can be viewed as a \mrba
of weight $\lambda$ by taking $\Omega$ to be a singleton.

\item An earlier notion of algebras with multiple Rota-Baxter operators, called a Rota-Baxter family algebra, was defined in~\mcite{EGP07}. See also~\mcite{Guo09,ZGM}.
The difference is that there the index set $\Omega$ is a semigroup and the  Rota-Baxter equation in Rota-Baxter family algebras is defined by
\begin{align*}
P_\alpha(x)P_{\beta}(y)=P_{\alpha\beta}(P_{\alpha}(x)y)+P_{\alpha\beta}(xP_{\beta}(y))+\lambda P_{\alpha\beta}(xy) \, \text{ for }\, x,y \in A, \alpha, \beta\in \Omega.
\end{align*}
\item
More recently, another notion of algebras with multiple operators, called a Rota-Baxter system, was introduced by Brzensi\`nski in~\cite{Brz16} to study dendriform algebras and covariant bialgebras. It is defined to be a triple $(R,P,S)$ consisting of an algebra $R$ and linear maps $P,Q: R\to R$ such that
$$P(x)P(y)=P(P(x)y)+P(xQ(y)), \quad Q(x)Q(y)=Q(P(x)y)+Q(xQ(y)), \quad x, y\in R.$$
See also Remark~\mref{rk:aybe}.
\end{enumerate}
\mlabel{rk:rbs}
\end{remark}

A natural example of Rota-Baxter operator of weight 0 is the operator of Riemann integral. Consider the $\RR$-algebra $R:=\mathrm{Cont}(\RR)$ of continuous functions on $\RR$. Then the linear operator

\begin{equation}
I: R\longrightarrow R, \quad f(x)\mapsto \int_0^x f(t)\,dt, \quad f\in R,
\mlabel{eq:int}
\end{equation}
is a Rota-Baxter operator of weight 0~\mcite{Bax60}.

We can generalize this to the multiple case as follows:

\begin{exam}
Let $R:=\mathrm{Cont}(\RR)$ be as above. Fix a family $k_\omega(x)$ (called kernels in integral equations~\mcite{GGL,Ze}) of functions in $R$ parameterized by $\omega\in \Omega$. Define
\begin{equation} I_\omega: R\longrightarrow R, \quad
f(x)\mapsto \int_0^x k_\omega(t)f(t)\,dt, \quad \omega \in \Omega.
\mlabel{eq:mint}
\end{equation}

Note that $I_\omega(f)=I(k_\omega f)$ for the integral operator $I$ in Eq.~\eqref{eq:int}.
Then from the Rota-Baxter property of $I$, we obtain, for $\alpha, \beta\in \Omega$ and $f, g\in R$,
$$
I_\alpha(f)I_\beta(g)
= I(k_\alpha f)I(k_\beta g) = I(k_\alpha f I(k_\beta g))+I(I(k_\alpha f) k_\beta g) = I_\alpha(f I_\beta(g)) + I_\beta (I_\alpha (f) g).
$$
Thus $(R,\{I_\omega\,|\,\omega\in \Omega\})$ is a \match \rba of weight zero.
\mlabel{ex:int}
\end{exam}

In fact, the same argument works more generally. Let $(R,P)$ be a Rota-Baxter algebra of weight zero and let $\{k_\omega\,|\,\omega\in \Omega\}\subseteq R$. Then the same argument shows that $(R,\{P_{k_\omega}\,|\,\omega\in \Omega\})$, where $P_{k_\omega}(x):=P(k_\omega x)$, is a \match \rba of weight zero.
However this setup no longer works when the weight is not zero.

\begin{exam}\cite[Lemma~3.7]{Guo08}
Let $(R, P)$ be a Rota-Baxter algebra of weight 0. Define $\tilde{P}:=-P$. Then $\tilde{P}$ is also a Rota-Baxter operator of weight 0 and the pair $(R, \{P, \tilde{P}\})$ is a \mrba of weight 0.
\end{exam}

The following result shows that a \match Rota-Baxter algebra gives rise to other \match Rota-Baxter algebras by arbitrary finite linear combinations.

For a map $A: \Omega\to \bfk$, $\omega \mapsto a_\omega$ with finite support, we also denote $A=\{a_\omega\,|\,\omega\in \Omega\}$.

\begin{prop}
Let $(R,\, \{P_{\omega}\mid \omega \in \Omega\})$ be a \match \rba of weight $\lambda_\Omega$. Let $A_i:\Omega\to \bfk, i\in I$, be a family of maps with finite supports, identified with $A_i=\{a_{i,\omega}\in \bfk\,|\,\omega\in \Omega\}$. Consider the linear combinations
\begin{align}
P_i:=P_{A_i}:=\sum_{\omega\in \Omega} a_{i,\omega} P_\omega.
\mlabel{eq:fav11}
\end{align}
Then
$(R,\,\{P_i\mid i\in I\})$ is a \match Rota-Baxter algebra of weight $\lambda_I:=\{\lambda_i:=\sum_{\omega\in \Omega} a_{i,\omega} \lambda_\omega\,|\, i\in I\}$. In particular, in a \match Rota-Baxter algebra, any linear combination of the Rota-Baxter operators $P_\omega$ is still a Rota-Baxter operator of a suitable weight.
\mlabel{lemm:twoopes}
\end{prop}
\begin{proof}
For $x, y\in R$ and $i, j \in I$, by Eq.~(\mref{eq:fav11}), we have
\begin{align*}
P_{i}(x)P_{j}(y)
=&\left(\sum_{\alpha\in \Omega} a_{i,\alpha} P_\alpha (x) \right) \left(\sum_{\beta\in \Omega} a_{j,\beta} P_\beta(y)\right)\\
=&\ \sum_{\alpha\in \Omega} \sum_{\beta\in \Omega} a_{i,\alpha} b_{j,\beta } P_\alpha(x)P_\beta(y)\\
=&\ \sum_{\alpha\in \Omega} \sum_{\beta\in \Omega} a_{i,\alpha} a_{j,\beta} \left(P_\alpha ( x P_\beta (y) )+P_\beta\left( P_\alpha (x)y\right) +\lambda_\beta P_\alpha(xy) \right)\quad \text{(by Eq.~(\mref{eq:RBid}))}\\
=&\ \sum_{\alpha\in \Omega} \sum_{\beta\in \Omega}
a_{i,\alpha} a_{j,\beta} P_\alpha ( x P_\beta (y))
+\sum_{\alpha\in \Omega}\sum_{\beta\in \Omega}a_{i,\alpha} a_{j,\beta} P_\beta\left( P_\alpha (x)y\right)+ \sum_{\alpha\in \Omega}\sum_{\beta\in \Omega}\lambda_\beta a_{i,\alpha} a_{j,\beta} P_\alpha(xy)\\
=&\ \sum_{\alpha\in \Omega} a_{i,\alpha} P_\alpha \left( x \sum_{\beta\in \Omega} a_{j,\beta} P_\beta (y) \right)+\sum_{\beta\in \Omega} a_{j,\beta} P_\beta\left(\sum_{\alpha\in \Omega} a_{i,\alpha} P_\alpha (x)y\right)+\left(\sum_{\beta\in \Omega}\lambda_\beta a_{j,\beta}\right) \sum_{\alpha\in \Omega} a_{i,\alpha} P_\alpha(xy)\\
=&\ \sum_{\alpha\in \Omega} a_{i,\alpha} P_\alpha \big( x P_j(y) \big)+\sum_{\beta\in \Omega} a_{j,\beta} P_\beta\big(P_i(x)y\big)+\left(\sum_{\beta\in \Omega}\lambda_\beta a_{j,\beta}\right) P_i(xy)\\
=&\ P_{i} (xP_{j}(y) )+P_{j}(P_{i} (x)y)+ \left(\sum_{\beta\in \Omega}\lambda_\beta a_{j,\beta}\right) P_{i}(xy).
\end{align*}
This proves the first statement. In particular, for any $i\in I$, $(R,\,P_i)$ is a  Rota-Baxter algebra of weight $ \sum_{\omega\in \Omega}\lambda_\omega a_{i,\omega}$.
\end{proof}

\subsection{Polarized associative Yang-Baxter equations}
As noted in the introduction, one of our motivations for the notion of \mrbas comes from associative Yang-Baxter equations. Let us first recall the concept of weighted associative Yang-Baxter equations~\mcite{OP10,ZGZ18}.

Let $R$ be a unitary algebra. For an element $r=\sum_{i}u_i\ot v_i \in R\ot R$, we write
\begin{align*}
r_{12}=\sum_{i}u_i\ot v_i \ot 1, \, r_{13}=\sum_{i}u_i\ot 1\ot v_i , \, r_{23}=\sum_{i}1\ot u_i \ot v_i.
\end{align*}
Generalizing the notion of the associative Yang-Baxter equation introduced by Aguiar~\mcite{Agu01}, we have

\begin{defn}~\cite{OP10, ZGZ18}
\label{def:WAYBE}
Let $\lambda$ be a given element of $\bfk$ and $R$ a unitary algebra.
\begin{enumerate}
\item The equation
\begin{align}
r_{13}r_{12}-r_{12}r_{23}+r_{23}r_{13}=-\lambda r_{13} \mlabel{eq:AYBE}
\end{align}
is called the {\bf  associative Yang-Baxter equation (AYBE) of weight $\lambda$}.
\item
An element $r=\sum_{i}u_i\ot v_i \in R\ot R$ is called a {\bf solution of the associative Yang-Baxter equation of weight $\lambda$} in $R$ if it satisfies  Eq.~(\mref{eq:AYBE}).
\end{enumerate}
\end{defn}

\begin{remark}\label{rem:AYBE}
\begin{enumerate}
\item The Eq.~(\mref{eq:AYBE}) was first studied by Ogievetsky and Popov~\mcite{OP10} by the name of non-homogenous associative classical Yang-Baxter equation. They found that a solution $r\in R\ot R$ of the non-homogenous associative classical Yang-Baxter equation induces an interesting algebraic structure, called a weighted infinitesimal unitary bialgebra~\mcite{GZ, ZCGL18, ZGZ18, ZZL19}.

\item An analogous of associative Yang-Baxter equations was introduced in~\mcite{BGN1} given by
\begin{align*}
r_{12}r_{13}+r_{13}r_{23}-r_{23}r_{12}=\lambda (r_{13}+r_{31})(r_{23}+r_{32}),
\end{align*}
which was called an extended associative Yang-Baxter equation of mass $\lambda$.
\end{enumerate}
\end{remark}

The following result captures the relationship between an AYBE of weight $\lambda$ and a Rota-Baxter operator of weight $\lambda$, generalizing the relationship established in~\mcite{Agu00}.

\begin{prop}~\cite[Theorem~4.10]{ZGZ18}
Let $r=\sum_{i}u_i\ot v_i$ be a solution of the AYBE of weight $\lambda$ in $R$. Then the linear operator
\begin{align*}
P_r: R\rightarrow R, \quad x\mapsto P_r(x):=\sum_{i}u_i xv_i,
\end{align*}
is a Rota-Baxter operator of weight $\lambda$.
\mlabel{pro:AYBE}
\end{prop}

In order to investigate the relationship between the associative Yang-Baxter equation and the classical Yang-Baxter equation, Aguiar~\mcite{Agu01} introduced a polarized form of the expression on the left hand side of the associative Yang-Baxter equation in Eq.~(\ref{eq:AYBE}):
\begin{align}
\{r,s\} :=r_{13}s_{12}-r_{12}s_{23}+r_{23}s_{13},
\mlabel{eq:22AYBE}
\end{align}
where $r, s\in R\ot R$. We call the corresponding equation
\begin{equation}
r_{13}s_{12}-r_{12}s_{23}+r_{23}s_{13} = 0
\mlabel{eq:paybe}
\end{equation}
the {\bf \paybe}, or {\bf \cpaybe} in short. More generally, we propose the following notion.

\begin{defn}
Let $\lambda$ be a given element of $\bfk$ and $t, s\in R\ot R$.
\begin{enumerate}
\item The equation
\begin{align}
r_{13}s_{12}-r_{12}s_{23}+r_{23}s_{13}=-\lambda s_{13} \mlabel{eq:2AYBE}
\end{align}
is called the {\bf  \paybe (\cpaybe) of weight $\lambda$}.
\item
Let $\Omega$ be a parameter set. A parameterized family $\{r_\omega\,|\,\omega\in \Omega\}$ of elements in an associative algebra $R$ is called a {\bf solution of the \paybe} in $R$ if every pair $(r,s):=(r_\alpha,r_\beta), \alpha, \beta \in \Omega$, satisfies  Eq.~(\mref{eq:2AYBE}).
\end{enumerate}
\end{defn}

\begin{remark}
As noted in Remark~\mref{rk:rbs}, Brzezi\`nski~\cite{Brz16} gave an analog of  \paybe, namely,
    \begin{align*}
r_{13}r_{12}-r_{12}r_{23}+s_{23}r_{13}=0,
s_{13}r_{12}-s_{12}s_{23}+s_{23}s_{13}=0,
    \end{align*}
and studied its relationship with his notion of Rota-Baxter systems.
\mlabel{rk:aybe}
\end{remark}

Generalizing Proposition~\mref{pro:AYBE}, we derive a \match Rota-Baxter operator of weight $\lambda$ from a \paybe of weight $\lambda$.

\begin{prop}
Let $\{r_\omega:=\sum_i u_{\omega,i}\ot v_{\omega,i}\in R\ot R\,|\,\omega\in \Omega\}$ be a solution of the \paybe of weight $\lambda$.
If $r_{12}s_{23}=s_{12}r_{23}$, then the linear operators
\begin{align*}
&P_\omega: R\rightarrow R, \quad x\mapsto P_\omega(x):=\sum_{i}u_{\omega,i} xv_{\omega,i}
\end{align*}
determine a \match Rota-Baxter algebra of weight $\lambda$.
\mlabel{prop:YBandRB}
\end{prop}

\begin{proof}
The trilinear map
$$ p: R \times R \times R \to     R,\quad   (u,v, w)
  \mapsto uxvyw \, \text{ for }\, u, v, w, x, y \in R, $$
induces a linear map
$$
      h
  \colon  R \ot R \ot R
  \to     R,
  \quad   u \ot v \ot w
  \mapsto uxvyw.
$$
Since $(r, s)$ is a solution of the \paybe of weight $\lambda$,  we have
\begin{align*}
r_{13}s_{12}-r_{12}s_{23}+r_{23}s_{13}=-\lambda s_{13}.
\end{align*}
If  $r_{12}s_{23}=s_{12}r_{23}$, then
\begin{align*}
r_{13}s_{12}-s_{12}r_{23}+r_{23}s_{13}=-\lambda s_{13},
\end{align*}
that is,
\begin{align*}
\sum_{i,j}u_{\alpha, i} u_{\beta, j}\ot v_{\beta, j}\ot v_{\alpha, i}-\sum_{i,j}u_{\beta, j} \ot v_{\beta, j}  u_{\alpha, i}\ot v_{\alpha, i} +\sum_{i,j} u_{\beta, j}\ot u_{\alpha, i}\ot v_{\alpha, i} v_{\beta, j}=-\lambda \sum_{i}u_{\alpha, i}\ot 1\ot v_{\alpha, i}.
\end{align*}
Applying $h$ to both sides of the equation, we obtain
\begin{align*}
\sum_{i,j}u_{\alpha, i} u_{\beta, j}x v_{\beta, j}y v_{\alpha, i}-\sum_{i,j}u_{\beta, j} x v_{\beta, j}  u_{\alpha, i}y v_{\alpha, i} +\sum_{i,j} u_{\beta, j}x u_{\alpha, i}y v_{\alpha, i} v_{\beta, j}=-\lambda \sum_{i}u_{\alpha, i}xy v_{\alpha, i},
\end{align*}
giving the desired \match Rota-Baxter equation:
\begin{equation*}
P_\alpha(P_\beta(x)y)-P_\beta(x)P_\alpha(y)+P_\beta(xP_\alpha(y)) =-\lambda P_\beta(xy). \hspace{3cm} \qedhere
\end{equation*}
\end{proof}

\section{\Match dendriform algebras and \match tridendriform algebras}
\mlabel{sec:dend}

Motivated by the natural connection of Rota-Baxter algebras with dendriform (tridendriform) algebras, we generalize the notion of dendriform algebras to that of \match dendriform (tridendriform) algebras. We then use edge decorated trees to construct \match dendriform algebras.

\subsection{Definitions and relations with \match Rota-Baxter algebras}
We begin with the notion of a \match dendriform algebra, generalizing the well-known notion of Loday~\mcite{Lo00}.
\begin{defn}
Let $\Omega$ be a \nes. A {\bf \match dendriform algebra} or more precisely an {\bf $\Omega$-\match dendriform algebra}, is a module $D$ together with a family of binary operations $\{\prec_\omega, \succ_\omega\mid \omega\in \Omega\}$, such that, for $ x, y, z\in D$ and $\alpha,\beta\in \Omega$, there is
\begin{align}
(x\prec_{\alpha} y) \prec_{\beta} z=\ & x \prec_{\alpha} (y\prec_{\beta} z)
+x\prec_{ \beta} (y \succ_{\alpha} z), \mlabel{eq:ddf1} \\
(x\succ_{\alpha} y)\prec_{\beta} z=\ & x\succ_{\alpha} (y\prec_{\beta} z),\quad \quad \quad \quad \quad \ \ \ \ \mlabel{eq:ddf2} \\
(x\prec_{\beta} y)\succ_{ \alpha}z
 +(x\succ_{\alpha} y) \succ_{\beta}z=\ & x\succ_{\alpha}(y \succ_{\beta} z). \mlabel{eq:ddf3}
\end{align}
\mlabel{de:mdend}
\end{defn}
Loday and Ronco~\mcite{LR04} also introduced the concept of a tridendriform algebra in their study of polytopes and Koszul duality. More generally, the concept of a  tridendriform algebra of weight $\lambda$, also called $\lambda$-tridendriform algebra, was introduced by Burgunder and Ronco~\mcite{BR10} and has a close relationship with brace algebras, $\lambda$-tridendriform bialgebras and  $\lambda$-Gerstenhaber-Voronov algebras.

Generalizing the notion of tridendriform algebras (with weight) to multiple triples of binary operators, we propose the following notion.
\begin{defn}
Let $\Omega$ be a \nes. A {\bf  \match tridendriform algebra} is a module $T$ equipped with
a family of binary operations $\{\prec_{\omega}, \succ_\omega, \cdot_\omega \,\mid\, \omega \in \Omega\}$ such that, for $ x, y, z\in T$ and $\alpha,\beta\in \Omega$,
\begin{align}
(x\prec_{\alpha} y)\prec_{\beta} z=\ & x\prec_{\alpha}(y\prec_{{\beta}}z) +x\prec_{\beta}(y\succ_{{\alpha}}z)+ x\prec_{\alpha}(y\cdot_\beta z), \mlabel{eq:tdf1}\\
(x\succ_{\alpha} y)\prec_{\beta} z=\ & x\succ_{\alpha}(y\prec_{\beta} z), \mlabel{eq:tdf2}\\
x\succ_{\alpha} (y\succ_{\beta} z) = \ & (x\prec_{{\beta}}y)\succ_{\alpha}z+(x\succ_{{\alpha}}y)\succ_{\beta}z +(x\cdot_\beta y)\succ_{\alpha}z, \mlabel{eq:tdf3} \\
(x\succ_{\alpha} y)\cdot_\beta  z=\  & x\succ_{\alpha}(y \cdot_\beta  z), \mlabel{eq:tdf4}\\
(x\prec_{\alpha} y)\cdot_\beta  z= \ & x \cdot_\beta (y\succ_{\alpha} z), \mlabel{eq:tdf5}\\
(x\cdot_\alpha y)\prec_{\beta} z= \ & x\cdot_\alpha (y\prec_{\beta} z),\mlabel{eq:tdf6} \\
(x\cdot_\alpha y)\cdot_\beta  z= \ & x\cdot_\alpha  (y\cdot_\beta  z).  \mlabel{eq:tdf7}
\end{align}
\end{defn}

\begin{prop}
Let $\Omega$ and $I$ be nonempty sets and let $A_i:\Omega\to \bfk, i\in I$, be a family of maps with finite supports, identified with $A_i=\{a_{i,\omega}\in \bfk\,|\,\omega\in \Omega\}$.
\begin{enumerate}
\item
Let $(T,\, \{\prec_\omega, \succ_\omega,\cdot_\omega\,\mid\, \omega \in \Omega\})$ be a \match tridendriform algebra.  Consider the linear combinations
\begin{align}
\odot_i:=\odot_{A_i}:=\sum_{\omega\in \Omega} a_{i,\omega} \odot_\omega, \quad \odot \in \{\prec, \succ, \cdot\}, \quad i\in I.
\mlabel{eq:dfav}
\end{align}
Then
$(T,\,\{\prec_i, \succ_i, \cdot_i\,\mid\, i\in I\})$ is a \match tridendriform algebra. In particular, in a \match tridendriform algebra, any linear combination $\prec_i, \succ_i, \cdot_i$ (for a fixed $i$) still gives a tridendriform algebra structure on $T$.
\mlabel{it:mtd}
\item
The same result holds for a \match dendriform algebra.
\mlabel{it:md}
\end{enumerate}
\mlabel{prop:mmdd}
\end{prop}

\begin{proof}
We just verify Item~(\mref{it:mtd}). Item~(\mref{it:md}) can be verified in the same way. For $x, y\in T$ and $i, j \in I$, we have
\begin{align*}
&\ (x\prec_{i} y) \prec_{j} z\\
=&\ \sum_{\beta \in \Omega} a_{j, \beta} \left(\sum_{\alpha\in \Omega} a_{i,\alpha} x \prec_\alpha y\right) \prec_{\beta} z \quad \text{(by Eq.~(\mref{eq:dfav}))}\\
=&\ \sum_{\alpha\in \Omega}\sum_{\beta \in \Omega} a_{i,\alpha} a_{j, \beta} (x \prec_\alpha y)\prec_{\beta} z\\
=&\ \sum_{\alpha\in \Omega}\sum_{\beta \in \Omega} a_{i,\alpha} a_{j, \beta} \Big( x \prec_{\alpha} (y\prec_{\beta} z)
+x\prec_{ \beta} (y \succ_{\alpha} z)+x\prec_\alpha (y\cdot_\beta z)\Big)\quad \text{(by Eq.~(\mref{eq:tdf1}))}\\
=&\ \sum_{\alpha\in \Omega}\sum_{\beta \in \Omega} a_{i,\alpha} a_{j, \beta} x \prec_{\alpha} (y\prec_{\beta} z)
+\sum_{\alpha\in \Omega}\sum_{\beta \in \Omega} a_{i,\alpha} a_{j, \beta} x\prec_{ \beta} (y \succ_{\alpha} z)+\sum_{\alpha\in \Omega}\sum_{\beta \in \Omega} a_{i,\alpha} a_{j, \beta} x\prec_{ \alpha} (y \cdot_{\beta} z)\\
=&\ \sum_{\alpha\in \Omega} a_{i,\alpha} x \prec_{\alpha} \left(\sum_{\beta \in \Omega}a_{j, \beta} y\prec_{\beta} z\right)
+\sum_{\beta \in \Omega} a_{j, \beta} x\prec_{ \beta} \left(\sum_{\alpha\in \Omega}a_{i,\alpha} y \succ_{\alpha} z\right)+\sum_{\alpha\in \Omega} a_{i,\alpha} x \prec_{\alpha} \left(\sum_{\beta \in \Omega}a_{j, \beta} y\cdot_{\beta} z\right)\\
=&\ \sum_{\alpha\in \Omega} a_{i,\alpha} x \prec_{\alpha} \left(  y\prec_{j} z\right)
+\sum_{\beta \in \Omega} a_{j, \beta} x\prec_{ \beta} \left( y \succ_{i} z\right) +\sum_{\alpha\in \Omega} a_{i,\alpha} x \prec_{\alpha} \left(  y\cdot_{j} z\right)\quad \text{(by Eq.~(\ref{eq:dfav}) )}\\
=&\ x \prec_{i} (y\prec_{j} z)
+x\prec_{ j} (y \succ_{i} z)+x \prec_{i} (y\cdot_{j} z).
\end{align*}
By the same argument, we have
\begin{align*}
x\succ_{i} (y\succ_{j} z) = \ & (x\prec_{{j}}y)\succ_{i}z+(x\succ_{{i}}y)\succ_{j}z +(x\cdot_j y)\succ_{i}z.
\end{align*}
Further,
\begin{align*}
(x\succ_{i} y) \prec_{j} z=&
\ \sum_{\alpha\in \Omega} \sum_{\beta\in \Omega} a_{i,\alpha} a_{j, \beta} ( x \succ_\alpha y)\prec_\beta z\quad \text{(by Eq.~(\mref{eq:dfav}))}\\
=&\ \sum_{\alpha\in \Omega} \sum_{\beta\in \Omega} a_{i,\alpha} a_{j, \beta} x\succ_{\alpha} (y\prec_{\beta} z)\quad \text{(by Eq.~(\mref{eq:tdf2}))}\\
=&\ \sum_{\alpha\in \Omega} a_{i,\alpha} x\succ_{\alpha} \left(\sum_{\beta\in \Omega}a_{j, \beta} y\prec_{\beta} z\right)\\
=&\ \sum_{\alpha\in \Omega} a_{i,\alpha} x\succ_{\alpha} \left( y\prec_{j} z\right)\quad \text{(by Eq.~(\mref{eq:dfav}))}\\
=&\  x\succ_{i} ( y\prec_{j} z)\quad \text{(by Eq.~(\mref{eq:dfav}))}.
\end{align*}
We also have
\begin{align*}
(x\succ_{i} y)\cdot_j  z
=&\ \sum_{\alpha\in \Omega} \sum_{\beta\in \Omega} a_{i,\alpha} a_{j, \beta} ( x \succ_\alpha y)\cdot_\beta z\quad \text{(by Eq.~(\mref{eq:dfav}))}\\
=&\ \sum_{\alpha\in \Omega} \sum_{\beta\in \Omega} a_{i,\alpha} a_{j, \beta} x\succ_{\alpha} (y\cdot_{\beta} z)\quad \text{(by Eq.~(\mref{eq:tdf4}))}\\
=&\  x\succ_{i} ( y\cdot_{j} z)\quad \text{(by Eq.~(\mref{eq:dfav}))}.
\end{align*}
Similarly, we have
\begin{align*}
(x\cdot_i y)\prec_j z= \ & x\cdot_i (y\prec_j z).
\end{align*}
Moreover
\begin{align*}
(x\prec_{i} y)\cdot_j  z
=&\ \sum_{\alpha\in \Omega} \sum_{\beta\in \Omega} a_{i,\alpha} a_{j, \beta} ( x \prec_\alpha y)\cdot_\beta z\quad \text{(by Eq.~(\mref{eq:dfav}))}\\
=&\ \sum_{\alpha\in \Omega} \sum_{\beta\in \Omega} a_{i,\alpha} a_{j, \beta} x\cdot_\beta (y\succ_\alpha z)\quad \text{(by Eq.~(\mref{eq:tdf5}))}\\
=&\  x\cdot_j ( y\succ_{i} z)\quad \text{(by Eq.~(\mref{eq:dfav}))}.
\end{align*}
The last relation (Eq.~(\mref{eq:tdf7})) can be verified in the same way. In fact, it also follows from the property of matching associative algebras studied in Section~\mref{subsec:sc}. This completes the proof.
\end{proof}

We now establish the connections between \match \rbas, \match dendriform algebras and \match tridendriform algebras. For the classical case of one linear operator, see~\mcite{Agu01,E05}
\begin{theorem}
\begin{enumerate}
\item \label{it:RBTD2} A \match \rba $(R, \,\{P_{\omega} \mid \omega \in \Omega\} )$ of weight $\lambda_\Omega=\{\lambda_\omega\,|\,\omega\in \Omega\}$ induces a \match dendriform algebra $(R, \, \{\prec_{{\omega}}, \succ_{{\omega}}\, \mid \omega \in \Omega\})$, where
\begin{equation*}
x\prec_{{\omega}}y := xP_{\omega}(y)+\lambda_\omega xy,\,\, x\succ_{{\omega}}y := P_{\omega}(x)y, \, \,\text{ for }\, x,y \in R, \omega\in \Omega.
\end{equation*}
\item \label{it:RBTD1} A \match \rba $(R, \,\{P_{\omega} \mid \omega \in \Omega\} )$ of weight zero induces a \match dendriform algebra $(R, \, \{\prec_{{\omega}}, \succ_{{\omega}}  \, \mid \omega \in \Omega\})$, where
\begin{equation*}
x\prec_{{\omega}}y := xP_{\omega}(y),\,\, x\succ_{{\omega}}y := P_{\omega}(x)y, \,\text{ for }\, x,y \in R, \omega\in \Omega.
\end{equation*}
\item \label{thm:dend1}
A \match \rba $(R, \,\{P_{\omega} \mid \omega \in \Omega\} )$ of weight $\lambda_\Omega:=\{\lambda_\omega\,|\,\omega\in \Omega\}$ defines a \match tridendriform algebra
$(R, \, \{\prec_{{\omega}}, \succ_{{\omega}}, \cdot_\omega\, \mid \omega \in \Omega\})$, where
\begin{equation*}
x\prec_{{\omega}}y := xP_{\omega}(y),\,\, x\succ_{{\omega}}y := P_{\omega}(x)y, \,\, x\cdot_\omega y:=\lambda_\omega xy\,\,\text{ for }\, x, y \in R, \omega\in \Omega.
\mlabel{eq:RBTD2}
\end{equation*}

\end{enumerate}
\mlabel{thm:dend}
\end{theorem}
\begin{proof}
(\ref{it:RBTD2}). For $x, y, z\in R$ and $\alpha, \beta\in \Omega,$
\begin{align*}
(x\prec_{{\alpha}}y)\prec_{{\beta}}z
=&\ (xP_{\alpha}(y)+\lambda_\alpha xy)\prec_{{\beta}}z\\
=& \ (xP_{\alpha}(y)+\lambda_\alpha xy)P_{\beta}(z)+\lambda_\beta(xP_{\alpha}(y)
+\lambda_\alpha xy)z\\
=&\ xP_{\alpha}(y)P_{\beta}(z)+\lambda_\alpha xyP_{\beta}(z)+\lambda_\beta xP_{\alpha}(y)z+\lambda_\beta \lambda_\alpha xyz\\
=&\ x\left(P_{\alpha}(yP_{\beta}(z))+P_{\beta}(P_{\alpha}(y)z)
+\lambda_\beta P_\alpha(yz)\right) +\lambda_\alpha xyP_{\beta}(z)
+\lambda_\beta xP_{\alpha}(y)z+\lambda_\alpha\lambda_\beta xyz\\
=&\ xP_{\alpha}\Big(yP_{\beta}(z)+\lambda_\beta yz\Big)
+\lambda_\alpha xyP_{\beta}(z)
+\lambda_\alpha\lambda_\beta xyz+xP_{\beta}(P_{\alpha}(y)z)
+\lambda_\beta xP_{\alpha}(y)z\\
=&\  xP_{\alpha}\Big(yP_{\beta}(z)+\lambda_\beta yz\Big)+\lambda_\alpha x\Big(yP_{\beta}(z)+\lambda_\beta yz\Big)+\Big(xP_{\beta}(P_{\alpha}(y)z)+\lambda_\beta x(P_{\alpha}(y)z)\Big)\\
=& \ x\prec_{\alpha} (yP_{\beta}(z)+\lambda_\beta yz)+x\prec_{\beta} (P_{\alpha}(y)z)\\
=& \ x \prec_{\alpha} (y\prec_{\beta} z)
+x\prec_{ \beta} (y \succ_{\alpha} z),\\
(x\succ_{{\alpha}}y)\prec_{{\beta}}z=&\ (P_{\alpha}(x)y)\prec_{ \beta}z
=P_{\alpha}(x)yP_{\beta}(z)+\lambda_\beta P_{\alpha} (x)yz
=\ x\succ_{{\alpha}} (yP_{\beta}(z)+\lambda_\beta yz)
=x\succ_{{\alpha}}(y\prec_{{\beta}}z),\\
x\succ_{{\alpha}}(y\succ_{{\beta}}z) =&\ x\succ_{{\alpha}}(P_\beta (y)z)=P_{\alpha}(x)(P_{\beta}(y)z)=P_{\alpha}(x)P_{\beta}(y)z\\
=&\ P_{\alpha}(xP_{\beta}(y))z+P_{\beta}(P_{\alpha}(x)y)z+\lambda_\beta P_\alpha(xy)z\\
=&\ P_{\alpha}(xP_{\beta}(y)+\lambda_\beta xy)z+P_{\beta}(P_{\alpha}(x)y)z\\
= &\ (x\prec_{\beta} y)\succ_{ \alpha}z
 +(x\succ_{\alpha} y) \succ_{\beta}z,
\end{align*}
as desired.
\smallskip

\noindent
(\ref{it:RBTD1}). This follows from Item~(\ref{it:RBTD2}) by taking $\lambda=0$.
\smallskip

\noindent
(\ref{thm:dend1}).
For $x, y, z\in R$ and $\alpha, \beta\in \Omega,$
\begin{align*}
(x\prec_{{\alpha}}y)\prec_{{\beta}}z=\ &xP_{\alpha}(y)P_{\beta}(z)\\
=\ & x\Big(P_{\alpha}(yP_{\beta}(z))+P_{\beta}(P_{\alpha}(y)z)+\lambda_\beta P_{\alpha}(yz)\Big)\\
=\ & x\prec_{\alpha}(y\prec_{{\beta}}z) +x\prec_{\beta}(y\succ_{{\alpha}}z)+x\prec_{\alpha}(y\cdot_\beta z),\\
(x\succ_{{\alpha}}y)\prec_{{\beta}}z=\ &(P_{\alpha}(x)y)\prec_{{\beta}}z
=P_{\alpha}(x)(yP_{\beta}(z))
=x\succ_{{\alpha}}(y\prec_{{\beta}}z),\\
x\succ_{{\alpha}}(y\succ_{{\beta}}z)=\ & P_{\alpha}(x)P_{\beta}(y)z\\
=\ & \Big(P_{\alpha}(xP_{\beta}(y))+P_{\beta}(P_{\alpha}(x)y)+\lambda_\beta P_{\alpha}(xy)\Big)z\\
=\ & (x\prec_{{\beta}}y)\succ_{\alpha}z+(x\succ_{{\alpha}}y)\succ_{\beta}z+(x\cdot_\beta y)\succ_{\alpha}z,\\
(x\succ_{\alpha}y)\cdot_\beta z
=\ &(P_{\alpha}(x)y)\cdot_\beta z
=\lambda_\beta P_{\alpha}(x)yz
=P_{\alpha}(x)(\lambda_\beta yz)
= P_{\alpha}(x)(y\cdot_\beta z)
=x\succ_{\alpha}(y\cdot_\beta z),\\
(x\prec_{\alpha}y)\cdot_\beta z=\ &(xP_{\alpha}(y))\cdot_\beta z
=\lambda_\beta xP_{\alpha}(y)z
=\lambda_\beta x(P_{\alpha}(y)z)
= x\cdot_\beta (P_{\alpha}(y)z)
= x\cdot_\beta (y\succ_{\alpha}z),\\
(x\cdot_\beta y)\prec_{\alpha}z=\ &(\lambda_\beta xy)\prec_{\alpha}z
=(\lambda_\beta  xy) P_{\alpha}(z)
=\lambda_\beta  x(y P_{\alpha}(z))
= x\cdot_\beta (y\prec_{\alpha}z),\\
(x\cdot_\alpha y)\cdot_\beta z=\ &(\lambda_\alpha xy)\cdot_\beta z
=\lambda_\alpha \lambda_\beta (xyz)=\lambda_\alpha x(\lambda_\beta yz)
=x\cdot_\alpha(y\cdot_\beta z),
\end{align*}
as required.
\end{proof}

\subsection{\Match dendriform algebra on typed planar binary trees}
\mlabel{ss:tree}
The notion of a typed rooted tree was used to construct multiple (that is, \match) pre-Lie algebras~\mcite{Foi18} with motivation from renormalization of stochastic PDEs~\mcite{BHZ}. See Section~\mref{ss:mpl}. Here we apply a similarly defined typed planar binary trees to construct \match dendriform algebras, generalizing the construction of dendriform algebras on planar binary trees.

A ${\bf planar\ tree}$ is an oriented graph with an embedding into the plane and with  a preferred vertex called the $\mathbf{root}$. It is {\bf binary} when any vertex is trivalent (one root and two leaves)~\mcite{LR98, Gub}. The root is at the bottom of the tree. For each $n\geq 0$, the set of planar binary trees with $n$ interior
vertices will be denoted by $Y_n$. Let $\frakD$ be a nonempty set and let $\frakT$ be a set with at least two elements including a special element $e$.
For each $n\geq 0$, let $Y_{\mathfrak{D}, \mathfrak{T}}(n)$ denote the set of {\bf $\mathfrak{D}$-decorated $\mathfrak{T}$-typed planar binary trees} in $Y_n$, consisting of planar binary trees in $Y_n$ together with
\begin{enumerate}
\item
a decoration of the internal vertices by elements of $\mathfrak{D}$,
\item
a decoration of the internal edges (that is, connecting internal vertices) by elements of $\mathfrak{T}\backslash \{e\}$, and
\item
a decoration, usually suppressed, of the external/leaf edges (except the root edge) by $e$.
\end{enumerate}
The convenience of including $e$ as the {\bf empty types decoration} will become apparent below.
Adapting the drawing in~\mcite{BHZ,Foi18}, the following are some {\bf $\mathfrak{D}$-decorated $\mathfrak{T}$-typed planar binary trees} with two edge types $\textcolor{red}{|}$ (solid red) and  $\textcolor{green}{ \vdots}$ (dotted green) decorating the inner edges. Again the decoration by $e$ on the external edges are suppressed. Alternatively, one can simply decorate the edges by letters instead of colors or patterns.
\begin{align*}
Y_{\mathfrak{D}, \mathfrak{T}}(0)=\{|\},\ \
Y_{\mathfrak{D}, \mathfrak{T}}(1)&=\left\{ \YYY{\coordinate (a) at (-0.6,0.6);
\draw (o)--+(0,-1) (o)--+(1.2,1.2)
 (a)--+(-0.6,0.6);
\draw(o)--(a);
\path (o)+(-45:0.4) node {$a$}
(a)+(-135:0.4) node {}
(-0.05,0.5) node {};
},
\cdots \right\},\ \
Y_{\mathfrak{D}, \mathfrak{T}}(2)=\left\{ \YYY{\coordinate (a) at (-0.6,0.6);
\draw (o)--+(0,-1) (o)--+(1.2,1.2)
(a)--+(0.6,0.6) (a)--+(-0.6,0.6);
\draw[red](o)--(a);
\path (o)+(-45:0.4) node {$b$}
(a)+(-135:0.4) node {$a$}
(-0.05,0.5) node {};
}\, ,
\YYY{\coordinate (a) at (-0.6,0.6);
\draw (o)--+(0,-1) (o)--+(1.2,1.2)
(a)--+(0.6,0.6) (a)--+(-0.6,0.6);
\draw[green,densely dotted](o)--(a);
\path (o)+(-45:0.4) node {$b$}
(a)+(-135:0.4) node {$a$}
(-0.05,0.5) node {};
}\, ,
\YYY{\coordinate (b) at (0.6,0.6);
\draw (o)--+(0,-1) (o)--+(-1.2,1.2)
(b)--+(0.6,0.6) (b)--+(-0.6,0.6);
\draw[red](o)--(b);
\path (o)+(-45:0.4) node {$a$}
(b)+(-45:0.4) node {$b$};
}\, ,
\YYY{\coordinate (b) at (0.6,0.6);
\draw (o)--+(0,-1) (o)--+(-1.2,1.2)
(b)--+(0.6,0.6) (b)--+(-0.6,0.6);
\draw[green,densely dotted](o)--(b);
\path (o)+(-45:0.4) node {$a$}
(b)+(-45:0.4) node {$b$};
},
\cdots \right\},\\
Y_{\mathfrak{D}, \mathfrak{T}}(3)&=\left\{ \YYY{\coordinate (b) at (-0.5,0.5);
\coordinate (a) at (-0.9,0.9);
\draw (o)--+(0,-1) (o)--+(1.2,1.2) (b)--+(0.7,0.7)
(a)--+(0.3,0.3) (a)--+(-0.3,0.3);
\draw[red](o)--(a);
\path (o)+(-45:0.4) node {$c$}
(a)+(-135:0.4) node {$a$}
(b)+(-135:0.4) node {$b$}
(-0.5,0.9) node {}
(-0,0.5) node {};
}\, ,
\YYY{\coordinate (b) at (-0.5,0.5);
\coordinate (a) at (-0.9,0.9);
\draw (o)--+(0,-1) (o)--+(1.2,1.2) (b)--+(0.7,0.7)
(a)--+(0.3,0.3) (a)--+(-0.3,0.3);
\draw[green,densely dotted](o)--(a);
\path (o)+(-45:0.4) node {$c$}
(a)+(-135:0.4) node {$a$}
(b)+(-135:0.4) node {$b$}
(-0.5,0.9) node {}
(-0,0.5) node {};
}\, ,
\YYY{\coordinate (a) at (-0.7,0.7);
\coordinate (c) at (0.7,0.7);
\draw (o)--+(0,-1)
(a)--+(0.5,0.5) (a)--+(-0.5,0.5)
(c)--+(0.5,0.5) (c)--+(-0.5,0.5);
\draw[red](o)--(a);
\draw[green,densely dotted](o)--(c);
\path (o)+(-45:0.4) node {$b$}
(a)+(-135:0.4) node {$a$}
(c)+(-45:0.4) node {$c$};
}\, ,
\YYY{\coordinate (a) at (-0.7,0.7);
\coordinate (c) at (0.7,0.7);
\draw (o)--+(0,-1)
(a)--+(0.5,0.5) (a)--+(-0.5,0.5)
(c)--+(0.5,0.5) (c)--+(-0.5,0.5);
\draw[red](o)--(a);
\draw[red](o)--(c);
\path (o)+(-45:0.4) node {$b$}
(a)+(-135:0.4) node {$a$}
(c)+(-45:0.4) node {$c$};
}\, ,
\YYY{\coordinate (a) at (-0.7,0.7);
\coordinate (c) at (0.7,0.7);
\draw (o)--+(0,-1)
(a)--+(0.5,0.5) (a)--+(-0.5,0.5)
(c)--+(0.5,0.5) (c)--+(-0.5,0.5);
\draw[green,densely dotted](o)--(a);
\draw[red](o)--(c);
\path (o)+(-45:0.4) node {$b$}
(a)+(-135:0.4) node {$a$}
(c)+(-45:0.4) node {$c$};
},
\cdots\right\},
\end{align*}
with $a, b, c \in \mathfrak{D}$ and with $|\in Y_{\frakD,\frakT}(0)$ standing for the unique tree with one leaf.

Consider the disjoint union and the resulting direct sum
$$Y_{\mathfrak{D}, \mathfrak{T}}:= \bigsqcup_{n\geq 1}Y_{\mathfrak{D}, \mathfrak{T}}(n)\,\text{ and }\,  \mathrm{DD}_{\mathfrak{D}, \mathfrak{T}}:= \bfk Y_{\mathfrak{D}, \mathfrak{T}} = \bigoplus_{n\geq 1} \bfk Y_{\mathfrak{D}, \mathfrak{T}}(n).$$

A $\mathfrak{D}$-decorated $\mathfrak{T}$-typed planar binary tree $T$ in $Y_{\mathfrak{D}, \mathfrak{T}}(n)$ is called an {\bf $n$-$\mathfrak{D}$-decorated $\mathfrak{T}$-typed planar binary tree} or {\bf $n$-tree} for simplicity.
The $\mathbf{depth}$ $\mathrm{dep}(T)$ of a $\mathfrak{D}$-decorated $\mathfrak{T}$-typed planar binary tree $T$ is the maximal length of linear chains from the root to the leaves of the tree.
For example,
$$\dep(|)=0\,\text{ and }\, \dep({\YY{\node[scale=0.8] at (0.75,-0.55) {{$a$}};}}) = 1.$$

Let $T\in Y_{\mathfrak{D}, \mathfrak{T}}(m)$ and  $T'\in Y_{\mathfrak{D}, \mathfrak{T}}(n)$ be $\mathfrak{D}$-decorated $\mathfrak{T}$-typed planar binary trees,  $d \in \mathfrak{D}$ and $t_1, t_2 \in \mathfrak{T}$.  The
{\bf typed grafting}  $\vee_{d, t_1,t_2}$ of $T$ and $T'$ on $d$ is the $(n + m + 1)$-tree $T\vee_{d, t_1,t_2} T'\in Y_{\mathfrak{D}, \mathfrak{T}}(m+n+1)$,
obtained by joining the roots of $T$ and $T'$ to a new root decorated by $d$, via two new edges decorated by $t_1$ on the left edge and $t_2$ on the right one.
Then for any $\mathfrak{D}$-decorated $\mathfrak{T}$-typed planar binary tree $T\in Y_{\mathfrak{D}, \mathfrak{T}}(n)$ with $n\geq 1$, there exist unique elements
$T^l\in Y_{\mathfrak{D}, \mathfrak{T}}(k)$, $T^r\in Y_{\mathfrak{D}, \mathfrak{T}}(n-k-1)$ , $d\in \mathfrak{D}$,  and $t_1, t_2 \in \frakT$ such that
\begin{align*}
T=T^l\vee_{d, t_1, t_2}T^r,
\end{align*}
where $T^l$ and $T^r$ are the left (resp. right) branch of $T$. As noted above, the empty type decoration $e\in \frakT$ only decorates the leaf edges. For instance,
$$| \vee_{a, e,e}|=\YYY{\coordinate (a) at (-0.6,0.6);
\draw (o)--+(0,-1) (o)--+(1.2,1.2)
 (a)--+(-0.6,0.6);
\draw(o)--(a);
\path (o)+(-45:0.4) node {$a$}
(a)+(-135:0.4) node {}
(-0.05,0.5) node {};
} \, , \YYY{\coordinate (a) at (-0.6,0.6);
\draw (o)--+(0,-1) (o)--+(1.2,1.2)
 (a)--+(-0.6,0.6);
\draw(o)--(a);
\path (o)+(-45:0.4) node {$a$}
(a)+(-135:0.4) node {}
(-0.05,0.5) node {};
}\vee_{b, t_1,e}|  =\YYY{\coordinate (a) at (-0.6,0.6);
\draw (o)--+(0,-1) (o)--+(1.2,1.2)
(a)--+(0.6,0.6) (a)--+(-0.6,0.6);
\draw[red](o)--(a);
\path (o)+(-45:0.4) node {$b$}
(a)+(-135:0.4) node {$a$}
(-0.05,0.5) node {};
}\, ,\YYY{\coordinate (a) at (-0.7,0.7);
\coordinate (c) at (0.7,0.7);
\draw (o)--+(0,-1)
(a)--+(0.5,0.5) (a)--+(-0.5,0.5)
(c)--+(0.5,0.5) (c)--+(-0.5,0.5);
\draw[red](o)--(a);
\draw[green,densely dotted](o)--(c);
\path (o)+(-45:0.4) node {$b$}
(a)+(-135:0.4) node {$a$}
(c)+(-45:0.4) node {$c$};
}= \YYY{\coordinate (a) at (-0.6,0.6);
\draw (o)--+(0,-1) (o)--+(1.2,1.2)
 (a)--+(-0.6,0.6);
\draw(o)--(a);
\path (o)+(-45:0.4) node {$a$}
(a)+(-135:0.4) node {}
(-0.05,0.5) node {};
}\vee_{b, t_1, t_2}\YYY{\coordinate (a) at (-0.6,0.6);
\draw (o)--+(0,-1) (o)--+(1.2,1.2)
 (a)--+(-0.6,0.6);
\draw(o)--(a);
\path (o)+(-45:0.4) node {$c$}
(a)+(-135:0.4) node {}
(-0.05,0.5) node {};
},$$
where $t_1=\textcolor{red}{|}$ and  $t_2=\textcolor{green}{ \vdots}$ are two types and $a, b,c \in \mathfrak{D}$.

With these notions, we now proceed to construct a \match dendriform algebra on $\mathfrak{D}$-decorated $\Omega$-typed planar binary trees.

\begin{defn}
Let $\frakD$ and $\frakT$ be as defined above and let $\Omega=\frakT\backslash \{e\}$. Define a set $\{\prec_{\omega}, \succ_{\omega}  \, \mid \omega \in \Omega\}$ of binary operations  on $\mathrm{DD}_{\mathfrak{D}, \Omega}=\oplus_{n\geq 1} \bfk Y_{\mathfrak{D}, \Omega}(n)$ by the following recursion on the depth of trees.
\begin{enumerate}
\item $| \succ_{\omega}T=T\prec_{\omega}| =T$  and  $| \prec_{\omega}T=T\succ_{\omega}| =0$  for $\omega\in \Omega$  and $T \in Y_{\mathfrak{D}, \Omega}(n), n\geq 1$. \label{it:md61}
\item \label{it:md62}
For $T=T^l\vee_{a, \alpha, \beta}T^r$ and $U=U^l\vee_{b, \gamma, \delta}U^r$, define
    \begin{enumerate}
   \item \label{itt:md621}
    $T \prec_{\omega}U:= T^l \vee_{a, \alpha, \beta} (T^r\prec_\omega U)+T^l\vee_{a, \alpha, \omega}(T^r \succ_\beta U)$,
\item
$T \succ_{\omega}U:= (T\prec_\gamma U^l) \vee_{b, \omega, \delta}U^r+(T \succ_\omega U^l)\vee_{b, \gamma, \delta}U^r$.
    \end{enumerate}
\end{enumerate}\mlabel{def:md6}
\end{defn}

\begin{exam}
The following are two examples.
\begin{align*}
\YYY{\coordinate (a) at (-0.6,0.6);
\draw (o)--+(0,-1) (o)--+(1.2,1.2)
 (a)--+(-0.6,0.6);
\draw(o)--(a);
\path (o)+(-45:0.4) node {$a$}
(a)+(-135:0.4) node {}
(-0.05,0.5) node {};
} \prec_\alpha \YYY{\coordinate (a) at (-0.6,0.6);
\draw (o)--+(0,-1) (o)--+(1.2,1.2)
 (a)--+(-0.6,0.6);
\draw(o)--(a);
\path (o)+(-45:0.4) node {$b$}
(a)+(-135:0.4) node {}
(-0.05,0.5) node {};
}=\YYY{\coordinate (b) at (0.6,0.6);
\draw (o)--+(0,-1) (o)--+(-1.2,1.2)
(b)--+(0.6,0.6) (b)--+(-0.6,0.6);
\draw[red](o)--(b);
\path (o)+(-45:0.4) node {$a$}
(b)+(-45:0.4) node {$b$}
(0.05,0.5) node {$\alpha$};
}\, ,
\YYY{\coordinate (a) at (-0.6,0.6);
\draw (o)--+(0,-1) (o)--+(1.2,1.2)
 (a)--+(-0.6,0.6);
\draw(o)--(a);
\path (o)+(-45:0.4) node {$a$}
(a)+(-135:0.4) node {}
(-0.05,0.5) node {};
} \succ_\beta \YYY{\coordinate (a) at (-0.6,0.6);
\draw (o)--+(0,-1) (o)--+(1.2,1.2)
 (a)--+(-0.6,0.6);
\draw(o)--(a);
\path (o)+(-45:0.4) node {$b$}
(a)+(-135:0.4) node {}
(-0.05,0.5) node {};
}=\YYY{\coordinate (a) at (-0.6,0.6);
\draw (o)--+(0,-1) (o)--+(1.2,1.2)
(a)--+(0.6,0.6) (a)--+(-0.6,0.6);
\draw[green,densely dotted](o)--(a);
\path (o)+(-45:0.4) node {$b$}
(a)+(-135:0.4) node {$a$}
(-0.03,0.65) node {$\beta$};
},
\end{align*}
where $a, b  \in \mathfrak{D}$, $\alpha=\textcolor{red}{|}$ and $\beta=\textcolor{green}{ \vdots}$.
\end{exam}

\begin{prop}
Let $\frakD$ and $\Omega$ be nonempty sets. Then the pair $(\mathrm{DD}_{\mathfrak{D}, \Omega}, \{\prec_\omega, \succ_\omega\mid \omega \in \Omega\})$ is a \match dendriform algebra.
\end{prop}

\begin{proof}
For $\alpha, \beta \in \Omega$ and $S, T, U \in \mathrm{DD}_{\mathfrak{D}, \Omega}$, we may write
\begin{align*}
S:=S^l\vee_{a, \gamma, \delta}S^r\, , T:=T^l \vee_{b, \iota, \kappa}T^r\, , U:=U^l \vee_{c, \mu, \nu} U^r
\end{align*}
with $a, b, c \in \mathfrak{D}$, $\gamma, \delta, \iota, \kappa, \mu, \nu \in \Omega$ and $S^l, S^r, T^l, T^r, U^l, U^r \in \sqcup_{n\geq 0}Y_{\mathfrak{D}, \mathfrak{T}}(n)$. We prove this result by induction on $\dep (S)+\dep(T)+\dep(U)\geq 3$. For the initial step of $\dep (S)+\dep(T)+\dep(U)= 3$, we have $\dep (S)=\dep(T)=\dep(U)=1$ and so
\begin{align*}
S={\YY{\node[scale=0.8] at (0.75,-0.55) {{$a$}};}}\, , T={\YY{\node[scale=0.8] at (0.75,-0.55) {{$b$}};}}\, \text{ and }\, U={\YY{\node[scale=0.8] at (0.75,-0.55) {{$c$}};}}.
\end{align*}
Then
\begin{align*}
&\ (S\prec_\alpha T) \prec_\beta U\\
=&\ ({\YY{\node[scale=0.8] at (0.75,-0.55) {{$a$}};}}\prec_\alpha {\YY{\node[scale=0.8] at (0.75,-0.55) {{$b$}};}})
\prec_\beta {\YY{\node[scale=0.8] at (0.75,-0.55) {{$c$}};}}\\
=&\ \left(| \vee_{a, e, e}(| \prec_\alpha{\YY{\node[scale=0.8] at (0.75,-0.55) {{$b$}};}})+| \vee_{a, e, \alpha}(| \succ_e {\YY{\node[scale=0.8] at (0.75,-0.55) {{$b$}};}})\right)\prec_\beta {\YY{\node[scale=0.8] at (0.75,-0.55) {{$c$}};}}\quad \text{(by Definition~(\ref{def:md6})~(\ref{it:md62})~(i))}\\
=&\ \left(| \vee_{a, e, \alpha}{\YY{\node[scale=0.8] at (0.75,-0.55) {{$b$}};}}\right)\prec_\beta {\YY{\node[scale=0.8] at (0.75,-0.55) {{$c$}};}}\quad \text{(by Definition~(\ref{def:md6})~(\ref{it:md61}))}\\
=&\ | \vee_{a, e, \alpha}\left({\YY{\node[scale=0.8] at (0.75,-0.55) {{$b$}};}} \prec_\beta {\YY{\node[scale=0.8] at (0.75,-0.55) {{$c$}};}}\right) +| \vee_{a, e, \beta}\left({\YY{\node[scale=0.8] at (0.75,-0.55) {{$b$}};}} \succ_\alpha {\YY{\node[scale=0.8] at (0.75,-0.55) {{$c$}};}}\right)\quad \text{(by Definition~(\ref{def:md6})~(\ref{it:md62})~(i))}\\
=&\ | \vee_{a, e, e} \left(| \prec_\alpha \Big(\yb \prec_\beta \yc\Big)\right)
+ | \vee_{a, e, \alpha} \left(| \succ_e\Big(\yb \prec_\beta \yc \Big)\right)\\
&\ +  | \vee_{a, e, e} \left(| \prec_\beta \Big(\yb \succ_\alpha \yc\Big)\right)
+ | \vee_{a, e, \beta} \left(| \succ_e\Big(\yb \succ_\alpha \yc \Big)\right)\quad \text{(by Definition~(\ref{def:md6})~(\ref{it:md61}))}\\
=&\ (| \vee_{a, e, e} |)\prec_\alpha \left(\yb \prec_\beta \yc\right)+(| \vee_{a, e, e} |)\prec_\beta \left(\yb \succ_\alpha \yc\right)\quad \text{(by Definition~(\ref{def:md6})~(\ref{it:md62})~(i))}\\
=&\ \ya \prec_\alpha \left(\yb \prec_\beta \yc\right) +\ya \prec_\beta \left(\yb \succ_\alpha \yc\right)\\
=&\ S\prec_\alpha (T \prec_\beta U)+S \prec_\beta (T\succ_\alpha U).
\end{align*}
By the same argument, we have
$$
S \baa (T \bba U)= \left(\ya \lba  \yb \right)\baa \yc +\left(\ya \baa  \yb \right) \bba \yc.
$$
We also have
\begin{align*}
&\ (S\succ_\alpha T) \prec_\beta U\\
=&\ \left(\ya \baa \yb \right) \lba \yc\\
=&\ \left(\ya \baa (| \vee_{b, e,e} |) \right)\lba \yc\\
=&\ \left(\Big(\ya \prec_e |\Big) \vee_{b, \alpha, e } | + \Big(\ya \baa |\Big) \vee_{b, e, e} | \right) \lba \yc \quad \text{(by Definition~(\ref{def:md6})~(\ref{it:md62})~(ii))}\\
=&\ \left(\ya  \vee_{b, \alpha, e } |\right) \lba \yc \quad \text{(by Definition~(\ref{def:md6})~(\ref{it:md61}))}\\
=&\ \ya \vee_{b, \alpha, e} \left(| \lba \yc\right) +\ya \vee_{b, \alpha, \beta} \left(| \succ_e \yc \right)\quad \text{(by Definition~(\ref{def:md6})~(\ref{it:md62})~(i))}\\
=&\ \ya \vee_{b, \alpha, \beta} \yc\quad \text{(by Definition~(\ref{def:md6})~(\ref{it:md61}))}\\
=&\ \left(\ya \prec_e |\right)\vee_{b, \alpha, \beta} \yc +\left(\ya \baa |\right)\vee_{b, e, \beta} \yc \quad \text{(by Definition~(\ref{def:md6})~(\ref{it:md61}))}\\
=&\ \ya \baa \left( | \vee_{b, e, \beta}\yc\right)\quad \text{(by Definition~(\ref{def:md6})~(\ref{it:md62})~(ii))}\\
=&\ \ya \baa \left( | \vee_{b, e, e} \Big(| \lba \yc \Big)+ | \vee_{b, e, \beta}\Big(| \succ_e \yc \Big)\right)\quad \text{(by Definition~(\ref{def:md6})~(\ref{it:md61}))}\\
=&\ \ya \baa  \left(\yb \lba \yc\right) \quad \text{(by Definition~(\ref{def:md6})~(\ref{it:md62})~(i))}\\
=&\ S\baa (T \lba U).
\end{align*}

For the induction step of $\dep (S)+\dep(T)+\dep(U)> 3$, we have
\begin{align*}
&\ (S \laa T) \lba U\\
=&\ \Big((S^l\vee_{a, \gamma, \delta}S^r \laa T\Big) \lba U\\
=&\ \Big( S^l \vee_{a, \gamma, \delta}(S^r \laa T)+ S^l \vee_{a, \gamma,\alpha}(S^r \bdt T)\Big) \lba U\quad \text{(by Definition~(\ref{def:md6})~(\ref{it:md62})~(i))}\\
=&\ \Big(S^l \vee_{a, \gamma, \delta}(S^r \laa T)\Big) \lba U
+\Big(S^l \vee_{a, \gamma,\alpha}(S^r \bdt T)\Big) \lba U\\
=&\ S^l \vee_{a, \gamma, \delta} \Big((S^r \laa T)\lba U\Big)+S^l \vee_{a, \gamma, \beta}\Big((S^r \laa T)\bdt U\Big)\\
&\ + S^l \vee_{a, \gamma, \alpha} \Big((S^r \bdt T)\lba U\Big)+S^l \vee_{a, \gamma, \beta}\Big((S^r \bdt T)\baa U\Big)\quad \text{(by Definition~(\ref{def:md6})~(\ref{it:md62})~(i))}\\
=&\ S^l \vee_{a, \gamma, \delta} \Big((S^r \laa T)\lba U\Big)+ S^l \vee_{a, \gamma, \alpha} \Big((S^r \bdt T)\lba U\Big)\\
&\ +S^l \vee_{a, \gamma, \beta}\Big((S^r \laa T)\bdt U+(S^r \bdt T)\baa U
\Big)\\
=&\ S^l \vee_{a, \gamma, \delta}\Big(S^r \laa (T\lba U)+S^r\lba (T \baa U)\Big)+S^l \vee_{a, \gamma, \alpha} \Big(S^r \bdt (T \lba U)\Big)\\
&\ + S^l \vee_{a, \gamma, \beta} \Big( S^r \bdt (T\baa U)\Big)\quad \text{(by the induction hypothesis and Eqs.~(\ref{eq:ddf1})-(\ref{eq:ddf3}))}\\
=&\ S^l \vee_{a, \gamma, \delta}\Big( S^r \laa (T\lba U)\Big)+S^l \vee_{a, \gamma, \alpha} \Big(S^r \bdt (T \lba U)\Big)\\
&\ + S^l \vee_{a, \gamma, \delta}\Big(S^r\lba (T \baa U)\Big)+S^l \vee_{a, \gamma, \beta} \Big( S^r \bdt (T\baa U)\Big)\\
=&\ \Big(S^l \vee_{a, \gamma, \delta} S^r\Big)\laa (T\lba U)+\Big(S^l \vee_{a, \gamma, \delta} S^r\Big) \lba (T \baa U)\quad \text{(by Definition~(\ref{def:md6})~(\ref{it:md62})~(i))}\\
=&\ S\laa (T \lba U)+S\lba (T \baa U).
\end{align*}

By the same argument,
$$ S\baa (T \bba U)= (S\lba T) \baa U+(S \baa T) \bba U.
$$

We also have
\begin{align*}
&\ (S\baa T)\lba U\\
=&\ \Big(S\baa (T^l \vee_{b, \iota, \kappa}T^r)\Big)\lba U\\
=&\   \Big((S\lia T^l)\vee_{b, \alpha, \kappa}T^r\Big)\lba U+\Big((S\baa T^l)\vee_{b, \iota, \kappa}T^r\Big)\lba U\quad \text{(by Definition~(\ref{def:md6})~(\ref{it:md62})~(ii))}\\
=&\   (S\lia T^l)\vee_{b, \alpha, \kappa}(T^r\lba U)+(S\lia T^l)\vee_{b, \alpha, \beta}(T^r\bka U)\\
&\ +(S\baa T^l) \vee_{b, \iota, \kappa}(T^r\lba U)+ (S\baa T^l) \vee_{b, \iota, \beta}(T^r\bka U)\quad \text{(by Definition~(\ref{def:md6})~(\ref{it:md62})~(i))}\\
=&\   (S\lia T^l)\vee_{b, \alpha, \kappa}(T^r\lba U)+(S\baa T^l) \vee_{b, \iota, \kappa}(T^r\lba U)\\
&\ +(S\lia T^l)\vee_{b, \alpha, \beta}(T^r\bka U)+ (S\baa T^l) \vee_{b, \iota, \beta}(T^r\bka U)\\
=&\ S\baa \Big(T^l \vee_{b, \iota, \kappa}(T^r \lba U)\Big)+S\baa \Big(T^l \vee_{b, \iota, \beta}(T^r \bka U)\Big)\quad \text{(by Definition~(\ref{def:md6})~(\ref{it:md62})~(ii))}\\
=&\ S\baa \Big((T^l \vee_{b, \iota,\kappa}T^r)\lba U\Big)\quad \text{(by Definition~(\ref{def:md6})~(\ref{it:md62})~(ii))}\\
=&\ S\baa (T \lba U).
\end{align*}

This completes the induction on the depth.
\end{proof}

In fact, the pair $(\mathrm{DD}_{\mathfrak{D}, \Omega}, \{\prec_\omega, \succ_\omega\mid \omega \in \Omega\})$ can be shown to be a free object in the category of \match dendriform algebras.

\subsection{Splitting of compatible associative algebras}\mlabel{subsec:sc}

Generalizing the notion of a matching (associative) dialgebra~\mcite{ZBG1} from two binary operations to multiple binary operations, we define

\begin{defn}
A {\bf \match associative algebra} is a $\bfk$-module $A$ equipped with
a family of binary operations $\{\bullet_{\omega}  \,\mid \omega \in \Omega\}$
 such that
 \begin{align*}
 (x\bullet_\alpha y)\bullet_\beta z =x\bullet_\alpha (y\bullet_\beta z) \tforall x,y,z \in A, \alpha, \beta\in \Omega.
 \end{align*}
 A \match associative algebra is called {\bf totally compatible} if it satisfies
   \begin{align*}
 (x\bullet_\alpha y)\bullet_\beta z =x\bullet_\beta (y\bullet_\alpha z) \tforall x,y ,z\in A, \alpha, \beta\in \Omega.
 \end{align*}
 In particular, for each $\omega\in \Omega$, the operation $\bullet_\omega$ is associative.
\end{defn}

More generally, we propose
\begin{defn}
A {\bf  compatible (multiple) associative algebra} is a $\bfk$-module $A$ equipped with
a family of binary operations $\{\bullet_{\omega}  \,\mid \omega \in \Omega\}$
 such that
 \begin{align}
 (x\bullet_\alpha y)\bullet_\beta z+(x\bullet_\beta y)\bullet_\alpha z=x\bullet_\alpha (y\bullet_\beta z)+x\bullet_\beta (y\bullet_\alpha z)\tforall x,y,z \in A, \alpha, \beta\in \Omega. \mlabel{eq:wm1}
 \end{align}
\end{defn}
As a consequence, for each $\omega\in \Omega$, the operation $\bullet_\omega$ is associative.

\begin{remark}
\begin{enumerate}
\item Every \match associative algebra is a compatible  multiple associative algebra unless the characteristic of $A$ is 2.

\item The motivation of Eq.~(\mref{eq:wm1}) comes from the compatibility of two associative algebras  studied in \mcite{OS06, OS062, ZBG2}. More precisely, given two associative algebras $(A, \bullet_\alpha)$ and $(A, \bullet_\beta)$, they are called {\bf compatible} if, for any $a_\alpha, b_\beta\in \bfk $, the product
    \begin{align*}
    x \ast y =a_\alpha x \bullet_\alpha y+ b_\beta x \bullet_\beta y
    \end{align*}
    defines an associative algebra. Direct calculation shows that this condition is equivalent to Eq.~(\mref{eq:wm1}) by taking $a_\alpha= b_\beta=1$.
\end{enumerate}
\end{remark}

Generalizing the fact that a dendriform algebra induces an associative algebra, we have

\begin{theorem}[{\bf Splitting the  compatible multiple associativity}]
\begin{enumerate}
\item Let $(R, \, \{\prec_{{\omega}},\succ_{{\omega}}  \, \mid \omega \in \Omega\})$ be a \match dendriform algebra. Define a set of binary operations
\begin{align*}
\bullet_\omega: R\ot R \rightarrow R,\quad  x \bullet_\omega y:= x\succ_\omega y+x\prec_\omega y\, \text{ for } x,y,z\in R, \omega\in \Omega.
\end{align*}
Then the pair $(R, \{\bullet_\omega\mid \omega\in \Omega\})$ is a compatible associative algebra.
\mlabel{it:addid0}
\item Let $(R, \, \{\prec_{{\omega}},\succ_{{\omega}},\cdot_\omega  \, \mid \omega \in \Omega\}, \cdot)$ be a \match tridendriform algebra. Define a set of binary operations
\begin{align*}
\bullet_\omega: R\ot R \rightarrow R,\quad  x \bullet_\omega y:= x\succ_\omega y+x\prec_\omega y+x\cdot_\omega y\, \text{ for } x,y,z\in R, \omega\in \Omega.
\end{align*}
Then the pair $(R, \{\bullet_\omega\mid \omega\in \Omega\})$ is a compatible associative algebra.
\mlabel{it:addidd0}
\end{enumerate}\mlabel{them:spl0}
\end{theorem}
We note that a \match dendriform algebra does not give a \match associative algebra as one might expect.

\begin{proof}
We will only prove Items (\mref{it:addidd0}). The proof of Item (\mref{it:addid0}) is similar. By the definition, for $x,y,z\in R, \alpha, \beta \in \Omega$, we have
\begin{align*}
x\bullet_\alpha (y\bullet_\beta z)=&\ x\bullet_\alpha (y\succ_\beta z+y\prec_\beta z+y\cdot_\beta z)\\
=&\ x\bullet_\alpha (y\succ_\beta z)+x\bullet_\alpha(y\prec_\beta z)+x\bullet_\alpha (y\cdot_\beta z)\\
=&\ x\succ_\alpha (y\succ_\beta z)+x\prec_\alpha (y\succ_\beta z)+x\cdot_\alpha (y\succ_\beta z)\\
&\ +x\succ_\alpha (y\prec_\beta z)+x\prec_\alpha (y\prec_\beta z)+ x\cdot_\alpha (y\prec_\beta z)\\
&\ +x\succ_\alpha (y\cdot_\beta z)+x\prec_\alpha (y\cdot_\beta z)+ x\cdot_\alpha (y\cdot_\beta z)\\
=&\ (x\prec_{{\beta}}y)\succ_{\alpha}z+(x\succ_{{\alpha}}y)\succ_{\beta}z+(x\cdot_\beta y)\succ_{\alpha}z+x\prec_\alpha (y\succ_\beta z)\\
&\ +x\cdot_\alpha (y\succ_\beta z)
+x\succ_\alpha (y\prec_\beta z)+x\prec_\alpha (y\prec_\beta z)+ x\cdot_\alpha (y\prec_\beta z)\\
&\ +x\succ_\alpha (y\cdot_\beta z)+x\prec_\alpha (y\cdot_\beta z)+ x\cdot_\alpha (y\cdot_\beta z)\quad \text{(by Eq.~(\mref{eq:tdf3}))}.
\end{align*}
Also,
\begin{align*}
(x\bullet_\alpha y)\bullet_\beta z=&\ (x\succ_\alpha y+x\prec_\alpha y+x\cdot_\alpha y)\bullet_\beta z\\
=&\ (x\succ_\alpha y)\bullet_\beta z+(x\prec_\alpha y)\bullet_\beta z+(x\cdot_\alpha y)\bullet_\beta z \\
=&\ (x\succ_\alpha y)\succ_\beta z+(x\succ_\alpha y)\prec_\beta z+(x\succ_\alpha y) \cdot_\beta z\\
&\ +(x\prec_\alpha y)\succ_\beta z+(x\prec_\alpha y)\prec_\beta z+(x\prec_\alpha y) \cdot_\beta z\\
&\ +(x\cdot_\alpha y)\succ_\beta z+(x\cdot_\alpha y)\prec_\beta z+(x\cdot_\alpha y) \cdot_\beta z\\
=&\ (x\succ_\alpha y)\succ_\beta z+(x\succ_\alpha y)\prec_\beta z+(x\succ_\alpha y) \cdot_\beta z
+(x\prec_\alpha y)\succ_\beta z\\
&\ +x\prec_{\alpha}(y\prec_{{\beta}}z) +x\prec_{\beta}(y\succ_{{\alpha}}z)+x\prec_{\alpha}(y\cdot_\beta z)
+(x\prec_\alpha y) \cdot_\beta z\\
&\ +(x\cdot_\alpha y)\succ_\beta z+(x\cdot_\alpha y)\prec_\beta z
+(x\cdot_\alpha y) \cdot_\beta z\quad \text{(by Eq.~(\mref{eq:tdf1}))}\\
=&\ (x\succ_\alpha y)\succ_\beta z
+x\succ_\alpha( y\prec_\beta z)
+x\succ_\alpha (y \cdot_\beta z)
+(x\prec_\alpha y)\succ_\beta z\\
&\ +x\prec_{\alpha}(y\prec_{\beta}z)
+x\prec_{\beta}(y\succ_{\alpha}z)
+x\prec_{\alpha}(y\cdot_\beta z)
+x\cdot_\beta (y\succ_\alpha  z)\\
&\ +(x\cdot_\alpha y)\succ_\beta z
+x\cdot_\alpha (y\prec_\beta z)
+x\cdot_\alpha (y \cdot_\beta z)\quad \text{(by Eqs.~(\ref{eq:tdf2}), (\ref{eq:tdf4})-(\ref{eq:tdf7}))}.
\end{align*}
Then
\begin{align*}
(x\bullet_\alpha y)\bullet_\beta z-x\bullet_\alpha (y\bullet_\beta z)
=&\ (x\prec_\alpha y)\succ_\beta z+x\prec_{\beta}(y\succ_{\alpha}z)
+x\cdot_\beta (y\succ_\alpha  z)
+(x\cdot_\alpha y)\succ_\beta z\\
&\ -x\prec_\alpha (y \succ_\beta z)-(x\prec_\beta y) \succ_\alpha z-(x\cdot_\beta y)\succ_\alpha z-x\cdot_\alpha (y\succ_\beta z).
\end{align*}
Similarly, we have
\begin{align*}
(x\bullet_\beta y)\bullet_\alpha z-x\bullet_\beta (y\bullet_\alpha z)=&\ (x\prec_\beta y)\succ_\alpha z+x\prec_{\alpha}(y\succ_{\beta}z)+x\cdot_\alpha (y\succ_\beta  z)
+(x\cdot_\beta y)\succ_\alpha z\\
&\ -x\prec_\beta (y \succ_\alpha z)-(x\prec_\alpha y) \succ_\beta z-(x\cdot_\alpha y)\succ_\beta z-x\cdot_\beta (y\succ_\alpha z).
\end{align*}
Thus
\begin{align*}
(x\bullet_\alpha y)\bullet_\beta z-x\bullet_\alpha (y\bullet_\beta z)+(x\bullet_\beta y)\bullet_\alpha z-x\bullet_\beta (y\bullet_\alpha z)=0,
\end{align*}
that is,
\begin{align*}
(x\bullet_\alpha y)\bullet_\beta z+(x\bullet_\beta y)\bullet_\alpha z=x\bullet_\alpha (y\bullet_\beta z)+x\bullet_\beta (y\bullet_\alpha z).
\end{align*}
This completes the proof.
\end{proof}

\section{Matching Rota-Baxter Lie algebras, \match pre-Lie algebras and \match PostLie algebras}
\mlabel{sec:prelie}

In this section, we study the \match type structure in the Lie algebra context. We first derive \match pre-Lie algebras introduced by Foissy (called multiple pre-Lie algebras)~\mcite{Foi18} from \match dendriform algebras and \match Rota-Baxter Lie algebras of weight zero. We further introduce a notion of \match PostLie algebras from its relationship with \match Rota-Baxter Lie algebra with weight and study the splitting property of \match pre-Lie algebras, matching PostLie algebras and \match Rota-Baxter Lie algebras.

\subsection{\Match pre-Lie algebras}
\mlabel{ss:mpl}

We first recall the concept of a \match pre-Lie algebra and its construction by typed rooted trees which was used by Bruned, Hairer and Zambotti~\mcite{BHZ,Foi18}  to give a description of a renormalisation procedure of stochastic PDEs.
We then establish the relationship between \match dendriform algebras and \match pre-Lie algebras. As a consequence, we derive a \match pre-Lie algebra from a \mrba of weight zero.

\begin{defn}~\mcite{Foi18}
Let $\Omega$ be a nonempty set. A {\bf \match (left) pre-Lie algebra}, called {\bf multiple pre-Lie} in~\cite{Foi18}, is a pair $(A, \{\ast_\omega\mid \omega\in \Omega\})$ consisting of a $\bfk$-module $A$ and a set of binary operations $\ast_\omega: A\ot A \rightarrow A$, $\omega\in \Omega$, that satisfy
\begin{align}
x\ast_\alpha (y\ast_\beta z)-(x\ast_\alpha y)\ast_\beta z=y\ast_\beta (x\ast_\alpha z)-(y\ast_\beta x)\ast_\alpha z\, \text{ for all }\, x,y,z\in A, \alpha, \beta \in \Omega.
\mlabel{eq:mpreid}
\end{align}
\end{defn}

\begin{remark}\label{remk:pre}
\begin{enumerate}
\item When $\Omega$ is a singleton set, the \match pre-Lie algebra is a pre-Lie algebra.
    In fact, for any $\omega\in \Omega$, $(A, \ast_\omega)$  is a pre-Lie algebra.
More generally, as in the case of \match dendriform algebras (Proposition~\mref{prop:mmdd}), linear combinations of operations in a \match pre-Lie algebras naturally give pre-Lie algebras. More precisely,
let $(A, \{\ast_\omega\mid \omega\in \Omega\})$ be a \match pre-Lie algebra. For any linear combination
\begin{align*}
\ast:=\sum_{\omega\in \Omega} c_\omega \ast_\omega,\,  c_\omega \in \bfk,
\end{align*}
with finite support,  $(A, \ast)$ is a pre-Lie algebra~\cite{Foi18}. \mlabel{reit:3}

\item As in the case of a (left) pre-Lie algebra which is also called a left symmetric algebra, the defining equation~(\ref{eq:mpreid}) does not change when the pair $(x,\alpha)$ is exchanged with the pair $(y,\beta)$.
\end{enumerate}
\end{remark}
We now briefly recall the example of \match pre-Lie algebras from~\mcite{BHZ,Foi18}.
\begin{exam}
Let $\mathfrak{D}$ and $\mathfrak{T}$ be nonempty sets. Recall~\mcite{BHZ, Foi18} that a {\bf $\mathfrak{D}$-decorated $\mathfrak{T}$-typed rooted tree} is a triple $(T, d, t)$ consisting of a rooted tree $T$, a (vertex) decorated map $d: V(T)\rightarrow \mathfrak{D}$, where $V(T)$ is the set of vertices of $T$, and an edge decoration map, called typed map $t: E(T)\rightarrow \mathfrak{T}$, where $E(T)$ is the set of edges of $T$.
When $\frakT$ is the couple $\{\text{solid red}, \text{dotted green}\}$, the first few typed rooted trees are listed below:
$$\tdunc{$c$},\,  \tddeuxc{$a$}{$b$}{red}, \, \tddeuxc{$a$}{$b$}{green, densely dotted}\, ,
\tdtroisunc{$a$}{$c$}{$b$}{red}{red},\,
\tdtroisunc{$a$}{$c$}{$b$}{green, densely dotted}{green, densely dotted},\,
\tdtroisunc{$a$}{$c$}{$b$}{red}{green,densely dotted},\,
\tdtroisdeuxc{$a$}{$b$}{$c$}{red}{red},\,
\tdtroisdeuxc{$a$}{$b$}{$c$}{red}{green, densely dotted},\,
\tdtroisdeuxc{$a$}{$b$}{$c$}{green, densely dotted}{red},\,
\tdtroisdeuxc{$a$}{$b$}{$c$}{green, densely dotted}{green, densely dotted},$$
where $a, b, c \in \mathfrak{D}$  and $\mathfrak{T}$ contains two elements, denoted by $\textcolor{red}{|}$ and  $\textcolor{green}{ \vdots}$, respectively.

For $n\geq 0$, denoted by $\mathbb{T}_{\mathfrak{D}, \mathfrak{T}}(n)$ the set of isoclasses of $\mathfrak{D}$-decorated $\mathfrak{T}$-typed trees $T$ such that $|V(T)|=n$. Define the disjoint union
\begin{align*}
\mathbb{T}_{\mathfrak{D}, \mathfrak{T}}:=\bigsqcup_{n\geq 0}\mathbb{T}_{\mathfrak{D}, \mathfrak{T}}(n).
\end{align*}
Denote by $\mathfrak{g}_{\mathfrak{D}, \mathfrak{T}}$ the space of $\mathfrak{T}$-typed trees whose vertices are decorated by elements of $\mathfrak{D}$. Motivated by the classical pre-Lie product on rooted tree, Foissy~\mcite{Foi18}  equipped $\mathfrak{g}_{\mathfrak{D}, \mathfrak{T}}$ with a \match pre-Lie algebraic structure, where the pre-Lie product $\ast_t$, $t\in \mathfrak{T}$ is defined by
\begin{align*}
T \ast_t T':= \sum_{v\in V(T)} T \ast_t^{(v)} T' \, \text{ for all }\, T, T' \in \mathbb{T}_{\mathfrak{D}, \mathfrak{T}}.
\end{align*}
Here $T \ast_t^{(v)} T'$ is denoted by the $\mathfrak{D}$-decorated $\mathfrak{T}$-typed tree obtained by grafting $T'$ on $v$. For example,
given two types of edges $\textcolor{red}{|}$ (solid red) and $\textcolor{green}{\vdots}$ (dotted green),  for $a$, $b$, $c\in \mathfrak{D}$, we have
\begin{align*}
\tddeuxc{$a$}{$b$}{red}\ast_{\textcolor{red}{|}} \tdunc{$c$}&=
\tdtroisunc{$a$}{$c$}{$b$}{red}{red}+\tdtroisdeuxc{$a$}{$b$}{$c$}{red}{red},&
\tddeuxc{$a$}{$b$}{red}\ast_{\textcolor{green}{\cdot}} \tdunc{$c$}&=
\tdtroisunc{$a$}{$c$}{$b$}{red}{green,dotted}+\tdtroisdeuxc{$a$}{$b$}{$c$}{red}{green, ,dotted}.
\end{align*}
\end{exam}

The following result captures the relationship between \match dendriform algebras and \match pre-Lie algebras, generalizing the well-known relationship between dendriform algebras and pre-Lie algebras~\mcite{Agu00}.

\begin{theorem}
Let $(R, \, \{\prec_{{\omega}},\succ_{{\omega}}  \, \mid \omega \in \Omega\})$ be a \match dendriform algebra. Define a set of binary operations
\begin{align}
\ast_\omega: R\ot R \rightarrow R,\quad  x \ast_\omega y:= x\succ_\omega y-y\prec_\omega x\, \text{ for } x,y\in R, \omega\in \Omega.
\mlabel{eq:mpre0}
\end{align}
Then the pair $(R, \{\ast_\omega\mid \omega\in \Omega\})$ is a \match pre-Lie algebra.
\mlabel{thm:denpre}
\end{theorem}

\begin{proof}
For $x, y,z \in R$, $\alpha, \beta \in \Omega$, we have
\begin{align*}
&\ x\ast_\alpha (y\ast_\beta z)-(x\ast_\alpha y)\ast_\beta z\\
=&\ x\succ_\alpha(y\succ_\beta z)-(y\succ_\beta z)\prec_\alpha x
-x\succ_\alpha(z\prec_\beta y)+(z\prec_\beta y)\prec_\alpha x\\
&\ -(x\succ_\alpha y) \succ_\beta z+z\prec_\beta (x\succ_\alpha y)
+(y\prec_\alpha x)\succ_\beta z-z\prec_\beta(y\prec_\alpha x)\quad \text{(by Eq.~(\ref{eq:mpre0}))}\\
=&\ (x\prec_\beta y)\succ_\alpha z-(y\succ_\beta z)\prec_\alpha x
-x\succ_\alpha(z\prec_\beta y)+z\prec_\alpha(y\succ_\beta x)\\
&\ +z\prec_\beta (x\succ_\alpha y)+(y\prec_\alpha x)\succ_\beta z\quad \text{(by Eq.~(\ref{eq:ddf1}) and Eq.~(\ref{eq:ddf3}))}\\
=&\ (x\prec_\beta y)\succ_\alpha z+(y\prec_\alpha x)\succ_\beta z-(y\succ_\beta z)\prec_\alpha x-x\succ_\alpha(z\prec_\beta y)\\
&\ +z\prec_\alpha(y\succ_\beta x)+z\prec_\beta (x\succ_\alpha y)\\
=&\ (x\prec_\beta y)\succ_\alpha z+(y\prec_\alpha x)\succ_\beta z-y\succ_\beta (z\prec_\alpha x)-x\succ_\alpha(z\prec_\beta y)\\
&\ +z\prec_\alpha(y\succ_\beta x)+z\prec_\beta (x\succ_\alpha y)\quad \text{(by Eq.~(\ref{eq:ddf2}))}.
\end{align*}
Observe that the last expression is invariant under the exchange of $(x,\alpha)$ and $(y,\beta)$. Hence
\begin{align*}
x\ast_\alpha (y\ast_\beta z)-(x\ast_\alpha y)\ast_\beta z=y\ast_\beta (x\ast_\alpha z)-(y\ast_\beta x)\ast_\alpha z.
\end{align*}
This completes the proof.
\end{proof}

As a consequence of Theorem~\mref{thm:denpre}, we have the following result, specializing to those in~\mcite{AB08,GS00} when $\Omega$ is a singleton.

\begin{coro}
\begin{enumerate}
\item \label{coro:rbmpre}
A \match \rba $(R, \{P_\omega\mid \omega\in \Omega\})$ of weight zero  defines  a \match pre-Lie algebra $(R, \{\ast_\omega\mid \omega\in \Omega\})$, where
\begin{align*}
x \ast_\omega y:= P_{\omega}(x)y-yP_{\omega}(x)\, \text{ for } x,y\in R, \omega\in \Omega.
\end{align*}

\item \label{coro:rbmpre2}
The \mrba $(R, \{P_\omega\mid \omega\in \Omega\})$ of weight $\lambda$ defines a \match pre-Lie algebra $(R, \{\ast_\omega\mid \omega\in \Omega\})$, where
\begin{align*}
x \ast_\omega y:= P_{\omega}(x)y-yP_{\omega}(x)-\lambda yx\, \text{ for } x,y\in R, \omega\in \Omega.
\end{align*}

\end{enumerate}
\end{coro}
\begin{proof}
(\ref{coro:rbmpre}). It follows from Theorem~\mref{thm:dend} (\ref{it:RBTD1}) and Theorem~\mref{thm:denpre}.

(\ref{coro:rbmpre2}). It follows from Theorem~\mref{thm:dend} (\ref{it:RBTD2}) and Theorem~\mref{thm:denpre}.
\end{proof}

Note that the \match Rota-Baxter equation of weight zero in Eq.~(\mref{eq:RBid}) is Lie compatible in the sense that anti-symmetrizing the \match Rota-Baxter algebra relation on an associative algebra gives the same relation on a Lie algebra.
Motivated by this connection,  we propose the concept of the \match Rota-Baxter Lie algebra. The situation is quite different for PostLie algebras. See Section~\mref{ss:mpol}.

\begin{defn}
Let $\Omega$ be a nonempty set and let $\lambda_\Omega:=\{\lambda_\omega\,|\,\omega\in \Omega\}\subseteq \bfk$. A {\bf \match Rota-Baxter Lie algebra} of weight $\lambda_\Omega$ is a  triple $(\mathfrak{g}, [,], \{P_\omega\mid \omega\in \Omega\})$ consisting of a Lie algebra $(\mathfrak{g}, [,])$ and a set of linear operators $P_\omega: \mathfrak{g} \rightarrow \mathfrak{g}$, $\omega\in \Omega$, that satisfy the {\bf \match Rota-Baxter equation of weight $\lambda_\Omega$}
\begin{align}
[P_{\alpha}(x), P_{\beta}(y)]
=P_{\alpha}([x, P_{\beta}(y)])+P_{\beta}([P_{\alpha}(x), y])+\lambda_\beta P_\alpha([x,y])\, \text{ for all }\, x, y\in \mathfrak{g}, \alpha, \beta\in \Omega.
\label{eq:mlieid}
\end{align}
\end{defn}

The Rota-Baxter equation on a Lie algebra is the operator form of the classical Yang-Baxter equation~\mcite{Bai, GS00,STS}. Similarly, there should be a close relationship between the \match Rota-Baxter equation in~(\mref{eq:mlieid}) with weight zero and the polarized classical Yang-Baxter equation, as a Lie algebra variation of the \paybe. This will be pursued in a separate paper.

Generalizing the connection establishing by Golubschik and Sokolov~\mcite{GS00} from Rota-Baxter Lie algebras of weight zero to pre-Lie algebras, we obtain

\begin{theorem}
Let $(\mathfrak{g}, [,], \{P_\omega\mid \omega\in \Omega\})$ be a \match Rota-Baxter Lie algebra of weight zero $($that is, $\lambda_\omega=0$ for all $\omega\in \Omega$$)$.
Define a set of binary operations
\begin{align}
\ast_\omega: \mathfrak{g} \ot \mathfrak{g}  \rightarrow \mathfrak{g},\quad  x \ast_\omega y:= [P_\omega(x), y] \, \text{ for }\, x,y\in \mathfrak{g}, \omega\in \Omega.
\mlabel{eq:mpre1}
\end{align}
Then the pair $(\mathfrak{g}, \{\ast_\omega\mid \omega\in \Omega\})$ is a \match pre-Lie algebra.
\mlabel{thm:rbmpre1}
\end{theorem}

\begin{proof}
For $x, y,z \in R$, $\alpha, \beta \in \Omega$, we have
\begin{align*}
&\ x\ast_\alpha (y\ast_\beta z)-(x\ast_\alpha y)\ast_\beta z\\
=&\ x\ast_\alpha [P_{\beta}(y), z]-[P_\alpha(x), y]\ast_\beta z\quad \text{(by Eq.~(\ref{eq:mpre1}))}\\
=&\ [P_{\alpha}(x), [P_{\beta}(y), z]]-[P_\beta([P_\alpha(x),y]), z]\quad \text{(by Eq.~(\ref{eq:mpre1}))}\\
=&\ [P_{\alpha}(x), [P_{\beta}(y), z]]-[[P_{\alpha}(x), P_{\beta}(y)], z]+[P_{\alpha}([x, P_{\beta}(y)]), z] \quad \text{(by Eq.~(\ref{eq:mlieid}))}\\
=&\ [P_{\alpha}(x), [P_{\beta}(y), z]]+[z, [P_{\alpha}(x), P_{\beta}(y)]]-[P_{\alpha}([P_{\beta}(y),x]), z] \\
=&\ -[P_\beta(y), [z ,P_\alpha(x)]]-[P_{\alpha}([P_{\beta}(y),x]), z] \quad \text{(by  Jacobi identity)}\\
=&\ [P_\beta(y), [P_\alpha(x), z]]-[P_{\alpha}([P_\beta(y), x]), z]\\
=&\ y\ast_\beta [P_\alpha(x), z]-[P_\beta(y), x]\ast_\alpha z\quad \text{(by Eq.~(\ref{eq:mpre1}))}\\
=&\ y\ast_\beta (x\ast_\alpha z)-(y\ast_\beta x)\ast_\alpha z \quad \text{(by Eq.~(\ref{eq:mpre1}))}.
\end{align*}
This completes the proof.
\end{proof}

\subsection{\Match PostLie algebras}
\mlabel{ss:mpol}

The notion of a PostLie algebra arising from an operad study~\mcite{Val07} and finding quite broad applications. See~\mcite{BGN0} for example. We now introduce the concept of a \match Lie algebra based on which we propose an analogs of PostLie algebra to multiple operations, called \match PostLie algebra. In order to extend the connection from tridendriform algebras to PostLie algebras to the multiple operation case, the notion of \match associative PostLie algebra is also introduced.

\begin{defn}
A {\bf \match Lie algebra} is a module  $\mathfrak{g}$ together with a family of binary operations
\begin{align*}
[\cdot,\cdot]_\omega: \mathfrak{g} \otimes \mathfrak{g} \rightarrow \mathfrak{g},\, (x,y) \longmapsto [x,y]_\omega, \omega\in \Omega,
\end{align*}
called  {\bf \match Lie brackets}, which satisfy the following conditions:
\begin{enumerate}
\item {\bf Alternativity}: $[x,x]_\omega =0$,
\item {\bf The \match Jacobi identity}: $[x,[y,z]_\beta]_\alpha+[y,[z,x]_\alpha]_\beta+[z,[x,y]_\alpha]_\beta=0$,
\end{enumerate}
 for all $x, y, z \in \mathfrak{g} $ and $\omega, \alpha, \beta \in \Omega$.
\end{defn}

\begin{remark}\label{remk:mmlie}
\begin{enumerate}
\item  When $\Omega$ is a singleton, a \match Lie algebra reduces to a Lie algebra. In fact, for any $\omega\in \Omega$, $(\mathfrak{g}, [\cdot,\cdot]_\omega)$  is a Lie algebra.
\item \label{it:r3} A \match associative algebra $(A, \{\bullet_\omega\mid \omega \in \Omega\})$ has a natural \match Lie algebraic structure with the \match Lie bracket given by
    \begin{align*}
    [x, y]_\omega:=x\bullet_\omega y-y \bullet_\omega x\, \text{ for all }\, x, y\in A, \omega \in \Omega.
    \end{align*}
\end{enumerate}
\end{remark}

Built on the notion of \match Lie algebras, we have
\begin{defn}
Let $\Omega$ be a nonempty set and $A$ a $\bfk$-module. A {\bf \match (left) PostLie algebra} is a couple $(A, \{[\,,\,]_\omega,\circ_\omega\mid \omega\in \Omega\})$ in which $(A,\{[\,,\,]_\omega\,|\,\omega\in \Omega\}$ is a \match Lie algebra and $\circ_\omega: A\ot A \rightarrow A, \omega\in \Omega$ is  a set of binary operations, satisfying
\begin{align}\label{eq:mplie1}
x\circ_\alpha (y\circ_\beta z)-(x\circ_\alpha y)\circ_\beta z-&y\circ_\beta (x\circ_\alpha z)+(y\circ_\beta x)\circ_\alpha z=[x, y]\circ_\alpha z,\\ \label{eq:mplie2}
x\circ_\alpha [y,z]_\beta=&[x\circ_\alpha y, z]_\beta+[y, x\circ_\alpha z]_\beta \tforall x,y,z\in A, \alpha, \beta \in \Omega.
\end{align}
\end{defn}

When the Lie bracket $[,]$ is zero, we have exactly a \match pre-Lie algebra.

It is well-known that a tridendriform algebra or a Rota-Baxter algebra of nonzero weight gives a PostLie algebra. Such a result for \match structures, that is, an analog of Theorem~\mref{thm:denpre} and Theorem~\mref{thm:rbmpre1} does not hold for a \match PostLie algebra. A simple indication of this phenomenon is that the antisymmetrization of a \match Rota-Baxter associative algebra of nonzero weight is not always a \match Rota-Baxter Lie algebra.

To remedy this situation, we replace the Lie algebra $(A,[\,,\,,])$ in a PostLie algebra by an associative algebra and replace
Eqs.~(\mref{eq:mplie1}) and (\mref{eq:mplie2}) by their associative analogs, and propose the following notion.

\begin{defn}
Let $\Omega$ be a nonempty set and $A$ a $\bfk$-module. A {\bf \match (left) associative PostLie algebra} is a pair $(A, \{\star_\omega, \circ_\omega\mid \omega\in \Omega\})$ consisting of a \match associative algebra $(A, \{\star_\omega\,|\,\omega\in \Omega\})$ and a set of binary operations $\circ_\omega: A\ot A \rightarrow A$, $\omega\in \Omega$, satisfying
\begin{align}
\label{eq:mplie3}
x\circ_\alpha (y\circ_\beta z)-(x\circ_\alpha y)\circ_\beta z-&y\circ_\beta (x\circ_\alpha z)+(y\circ_\beta x)\circ_\alpha z=(x\star_\beta y)\circ_\alpha z-(y\star_\alpha x)\circ_\beta z,\\ \label{eq:mplie4}
x\circ_\alpha (y\star_\beta z)-x\circ_\alpha(z\star_\beta y)=&(x\circ_\alpha y)\star_\beta z- z\star_\beta (x\circ_\alpha y)+y\star_\beta ( x\circ_\alpha z)-( x\circ_\alpha z)\star_\beta y,
\end{align}
for all $x,y,z\in A, \alpha, \beta \in \Omega.$
\end{defn}

By antisymmetrizing $\star_\omega$, a \match associative PostLie algebra is a \match PostLie algebra.

Similar to the close relationship between \match dendriform algebras and \match pre-Lie algebras. We also have an analogous result to capture the relationship between \match tridendriform algebras and \match associative PostLie algebras.

\begin{theorem}
Let $(A, \, \{\prec_{{\omega}}, \succ_{{\omega}}, \cdot_\omega  \, \mid \omega \in \Omega\})$ be a \match tridendriform algebra. Define
\begin{align}
x\star_\omega y:=x\cdot_\omega y,\quad  x \circ_\omega y:= x\succ_\omega y-y\prec_\omega x\, \text{ for } x,y\in R, \omega\in \Omega. \label{eq:mpplie}
\end{align}
Then the couple $(A, \{\star_\omega,\circ_\omega\mid \omega\in \Omega\})$ is a \match associative PostLie algebra.
\mlabel{thm:triposlie}
\end{theorem}

\begin{proof}
For $x, y,z \in R$, $\alpha, \beta \in \Omega$, we have
\begin{align*}
&\ x\circ_\alpha (y\circ_\beta z)-(x\circ_\alpha y)\circ_\beta z\\
=&\ x\succ_\alpha(y\succ_\beta z)-(y\succ_\beta z)\prec_\alpha x
-x\succ_\alpha(z\prec_\beta y)+(z\prec_\beta y)\prec_\alpha x\\
&\ -(x\succ_\alpha y) \succ_\beta z+z\prec_\beta (x\succ_\alpha y)
+(y\prec_\alpha x)\succ_\beta z-z\prec_\beta(y\prec_\alpha x)
\quad \text{(by Eq.~(\ref{eq:mpplie}))}\\
=&\ (x\prec_\beta y)\succ_\alpha z-(y\succ_\beta z)\prec_\alpha x
-x\succ_\alpha(z\prec_\beta y)+z\prec_\alpha(y\succ_\beta x)
+z\prec_\beta (x\succ_\alpha y)\\
&\ +(y\prec_\alpha x)\succ_\beta z
+(x\cdot_\beta y)\succ_\alpha z+z\prec_\beta(y\cdot_\alpha x)
\quad \text{(by Eq.~(\ref{eq:tdf1}) and Eq.~(\ref{eq:tdf3}))}\\
=&\ (x\prec_\beta y)\succ_\alpha z+(y\prec_\alpha x)\succ_\beta z-(y\succ_\beta z)\prec_\alpha x-x\succ_\alpha(z\prec_\beta y)\\
&\ +z\prec_\alpha(y\succ_\beta x)+z\prec_\beta (x\succ_\alpha y)+(x\cdot_\beta y)\succ_\alpha z+z\prec_\beta(y\cdot_\alpha x)\\
=&\ (x\prec_\beta y)\succ_\alpha z+(y\prec_\alpha x)\succ_\beta z-y\succ_\beta (z\prec_\alpha x)-x\succ_\alpha(z\prec_\beta y)\\
&\ +z\prec_\alpha(y\succ_\beta x)+z\prec_\beta (x\succ_\alpha y)+(x\cdot_\beta y)\succ_\alpha z+z\prec_\beta(y\cdot_\alpha x)\quad \text{(by Eq.~(\ref{eq:tdf2}))}.
\end{align*}
Similarly,
\begin{align*}
y\circ_\beta (x\circ_\alpha z)-(y\circ_\beta x)\circ_\alpha z=&(x\prec_\beta y)\succ_\alpha z+(y\prec_\alpha x)\succ_\beta z-y\succ_\beta (z\prec_\alpha x)-x\succ_\alpha(z\prec_\beta y)\\
&\ +z\prec_\alpha(y\succ_\beta x)+z\prec_\beta (x\succ_\alpha y)+(y\cdot_\alpha x)\succ_\beta z+z\prec_\alpha(x\cdot_\beta y).
\end{align*}
Thus
\begin{align*}
&x\circ_\alpha (y\circ_\beta z)-(x\circ_\alpha y)\circ_\beta z- y\circ_\beta (x\circ_\alpha z)
 +(y\circ_\beta x)\circ_\alpha z\\
=&\ (x\cdot_\beta y)\succ_\alpha z- z\prec_\alpha(x\cdot_\beta y)-
(y\cdot_\alpha x)\succ_\beta z+z\prec_\beta(y\cdot_\alpha x)\\
=&\  (x\cdot_\beta y) \circ_\alpha z-(y\cdot_\alpha x) \circ_\beta z\quad \text{(by Eq.~(\ref{eq:mpplie}))}\\
=&\  (x\star_\beta y) \circ_\alpha z-(y\star_\alpha x) \circ_\beta z,
\end{align*}
proving Eq.~(\mref{eq:mplie3}). Also, by Eq.~(\ref{eq:mpplie}),
\begin{align*}
&\ x\circ_\alpha (y\star_\beta z-z\star_\beta y)-(x\circ_\alpha y)\star_\beta z+ z\star_\beta (x\circ_\alpha y)-y\star_\beta ( x\circ_\alpha z)+( x\circ_\alpha z)\star_\beta y\\
=&\ x\circ_\alpha (y\cdot_\beta  z-z\cdot_\beta  y)-(x\circ_\alpha y)\cdot_\beta  z+ z\cdot_\beta  (x\circ_\alpha y)-y\cdot_\beta  ( x\circ_\alpha z)+( x\circ_\alpha z)\cdot_\beta  y\\
=&\ x\succ_\alpha (y\cdot_\beta  z-z\cdot_\beta  y)-(y\cdot_\beta  z-z\cdot_\beta  y)\prec_\alpha x-(x\succ_\alpha y-y\prec_\alpha x)\cdot_\beta  z\\
&\ +z\cdot_\beta  (x\succ_\alpha y
-y\prec_\alpha x)-y\cdot_\beta (x\succ_\alpha z-z\prec_\alpha x)+(x\succ_\alpha z-z\prec_\alpha x)\cdot_\beta  y\\
=&\ 0 \quad \text{(by Eqs.~(\ref{eq:tdf4})-(\ref{eq:tdf6}))},
\end{align*}
proving Eq.~(\mref{eq:mplie4}).
\end{proof}

As a consequence, we have

\begin{coro}
Let $(R, \,\{P_{\omega} \mid \omega \in \Omega\} )$ be a \mrba of weight $\lambda_\Omega=\{\lambda_\omega\,|\,\omega\in \Omega\}$. Define
\begin{align*}
x\star_\omega y:= \lambda_\omega x y,\quad  x \circ_\omega y:= P_{\omega}(x)y-yP_{\omega}(x)\, \text{ for } x,y\in R, \omega\in \Omega.
\end{align*}
Then the pair $(R, \{\star_\omega, \circ_\omega\mid \omega\in \Omega\})$ is a \match associative PostLie algebra.
\end{coro}

\begin{proof}
It follows from Theorem~\mref{thm:dend}~(\mref{thm:dend1}) and Theorem~\mref{thm:triposlie}.
\end{proof}

\subsection{Splitting of compatible Lie algebras}
In this subsection, we introduce the concept of compatible (multiple) Lie algebras. We then show that \match pre-Lie algebras give a splitting of compatible Lie algebras.

\begin{defn}
A {\bf compatible (multiple) Lie algebra} is a module  $\mathfrak{g}$ together with a family of binary operations
\begin{align*}
 \mathfrak{g} \otimes \mathfrak{g} \rightarrow \mathfrak{g},\, (x,y) \longmapsto [x,y]_\omega, \omega\in \Omega,
\end{align*}
called  {\bf compatible multiple Lie brackets}, which satisfy the following conditions:
\begin{enumerate}
\item {\bf Alternativity}: $[x,x]_\omega =0$,
\item {\bf The coupling Jacobi identity}:
\begin{align}
[x,[y,z]_\alpha]_\beta+[y,[z,x]_\alpha]_\beta+[z,[x,y]_\alpha]_\beta+
[x,[y,z]_\beta]_\alpha+[y,[z,x]_\beta]_\alpha+[z,[x,y]_\beta]_\alpha=0,\label{eq:copid}
\end{align}
\end{enumerate}
 for all $x, y, z \in \mathfrak{g} $ and $\omega, \alpha, \beta \in \Omega$.
\end{defn}

\begin{remark}
\begin{enumerate}
\item Every  matching  Lie algebra is a compatible Lie algebra.
\item When $\Omega=\{\alpha,\beta\}$, Eq.~(\mref{eq:copid}) was first introduced by Magri~\mcite{Mag78} in the study of integrable Hamiltonian equations and was called the coupling condition in~\mcite{Mag78}.
Eq.~(\mref{eq:copid}) also appeared in the work of Dotsenko and Khoroshkin~\mcite{DK07} defining the operad $\mathfrak{Lie}_2$ encoding two compatible Lie algebras.
\item Given two Lie algebras $(\mathfrak{g}, [,]_\alpha)$  and $(\mathfrak{g}, [,]_\beta)$. Define a new bracket $[,]: \mathfrak{g} \ot \mathfrak{g} \rightarrow \mathfrak{g}$ by taking
    \begin{align*}
    [x, y]:=a_\alpha[x, y]_\alpha+b_\beta[x, y]_\beta\, \text{ for some }\, a_\alpha, b_\alpha \in \bfk.
    \end{align*}
   Clearly, this new bracket is both skew symmetric and bilinear. Then the pair $(\mathfrak{g}, [,])$ is further a Lie algebra if $[, ]$ satisfies the Jacobi identity
   $[x,[y,z]]+[y,[z,x]]+[z,[x,y]]=0$. By a direct calculation,  we note that this
condition is equivalent to Eq.~(\mref{eq:copid}), see~\mcite{Str} for more details.
\end{enumerate}
\end{remark}

In parallel to the fact that the pre-Lie algebra is the splitting of the Lie algebra, we have

\begin{theorem}[{\bf Splitting the  compatible Lie algebras}]
 Let $(A, \{\ast_\omega\mid \omega\in \Omega\})$  be a  \match pre-Lie algebra. Define a set of binary operations
\begin{align}
[\,,\,]_\omega: A\ot A \rightarrow A,\quad [x, y]_\omega:= x\ast_\omega y-y\ast_\omega x\, \text{ for all } \, x,y,z\in R, \omega\in \Omega.
\mlabel{eq:spi2}
\end{align}
Then the pair $(A, \{[\,,\,]_\omega\mid \omega\in \Omega\})$ is a compatible  Lie algebra. \mlabel{them:aplcm}
\end{theorem}

We note that the statement does not hold when compatible Lie algebras is replaced by \match Lie algebras.

\begin{proof}
Eq.~(\mref{eq:spi2}) directly implies $[x, x]_\omega=0$ for $x\in A$. It remains to verify the coupling Jacobi identity.
For $x, y, z \in A$, $\alpha, \beta \in \Omega$, we have
\begin{align*}
[x,[y,z]_\alpha]_\beta=&\ [x, y\ast_\alpha z-z\ast_\alpha y ]_\beta\\
=&\ x \ast_\beta (y\ast_\alpha z-z\ast_\alpha y)-(y\ast_\alpha z-z\ast_\alpha y) \ast_\beta x\\
=&\ {x \ast_\beta (y\ast_\alpha z)}
-{x \ast_\beta(z\ast_\alpha y)}
-{(y\ast_\alpha z) \ast_\beta x}
+{(z\ast_\alpha y) \ast_\beta x}.
\end{align*}
We also have
\begin{align*}
[y,[z,x]_\alpha]_\beta=&\ {y \ast_\beta (z\ast_\alpha x)}
-{y \ast_\beta(x\ast_\alpha z)}
-{(z\ast_\alpha x) \ast_\beta y}
+{(x\ast_\alpha z) \ast_\beta y},\\
[z,[x,y]_\alpha]_\beta=&\ {z \ast_\beta (x\ast_\alpha y)}
-{z \ast_\beta(y\ast_\alpha x)}
-{(x\ast_\alpha y) \ast_\beta z}
+{(y\ast_\alpha x) \ast_\beta z},\\
[x,[y,z]_\beta]_\alpha =&\ {x \ast_\alpha (y\ast_\beta z)}
-{x \ast_\alpha(z\ast_\beta y)}
-{(y\ast_\beta z) \ast_\alpha x}
+{(z\ast_\beta y) \ast_\alpha x},\\
[y,[z,x]_\beta]_\alpha=&\ {y \ast_\alpha (z\ast_\beta x)}
-{y \ast_\alpha(x\ast_\beta z)}
-{(z\ast_\beta x) \ast_\alpha y}
+{(x\ast_\beta z) \ast_\alpha y},\\
[z,[x,y]_\beta]_\alpha=&\ {z \ast_\alpha (x\ast_\beta y)}
-{z \ast_\alpha(y\ast_\beta x)}
-{(x\ast_\beta y) \ast_\alpha z}
+{(y\ast_\beta x) \ast_\alpha z}.
\end{align*}
Putting all together, we obtain
\begin{align*}
&\ [x,[y,z]_\alpha]_\beta+[y,[z,x]_\alpha]_\beta+[z,[x,y]_\alpha]_\beta+
[x,[y,z]_\beta]_\alpha+[y,[z,x]_\beta]_\alpha+[z,[x,y]_\beta]_\alpha\\
=&\ {x \ast_\alpha (y\ast_\beta z)}
-{(x\ast_\alpha y) \ast_\beta z}
-{y \ast_\beta(x\ast_\alpha z)}
+{(y\ast_\beta x) \ast_\alpha z}\\
&\ +{y \ast_\alpha (z\ast_\beta x)}
-{(y\ast_\alpha z) \ast_\beta x}
-{z \ast_\beta(y\ast_\alpha x)}
+{(z\ast_\beta y) \ast_\alpha x}\\
&\ +{z \ast_\alpha (x\ast_\beta y)}
-{(z\ast_\alpha x) \ast_\beta y}
-{x \ast_\beta(z\ast_\alpha y)}
+{(x\ast_\beta z) \ast_\alpha y}\\
&\ +{x \ast_\beta (y\ast_\alpha z)}
-{(x\ast_\beta y) \ast_\alpha z}
-{y \ast_\alpha(x\ast_\beta z)}
+{(y\ast_\alpha x) \ast_\beta z}\\
&\ +{y \ast_\beta (z\ast_\alpha x)}
-{(y\ast_\beta z) \ast_\alpha x}
-{z \ast_\alpha(y\ast_\beta x)}
+{(z\ast_\alpha y) \ast_\beta x}\\
&\ +{z \ast_\beta (x\ast_\alpha y)}
-{(z\ast_\beta x) \ast_\alpha y}
-{x \ast_\alpha(z\ast_\beta y)}
+{(x\ast_\alpha z) \ast_\beta y}.
\end{align*}
Since $(A, \{\ast_\omega\mid \omega\in \Omega\})$  is a  \match pre-Lie algebra, each row on the right hand side of the equation is zero, giving what we need.
\end{proof}

The following result provides the connection between compatible associative algebras and  compatible Lie algebras.

We also note that, when the term compatible is replaced by \match in the theorem, the corresponding statement is not valid. Indeed, if $(A, \{\bullet_\omega\mid \omega\in \Omega\})$  is a \match associative algebra, we still get a compatible Lie algebra, not a \match Lie algebra.

\begin{theorem}
Let $(A, \{\bullet_\omega\mid \omega\in \Omega\})$  be a  compatible associative algebras. Define a set of binary operations
\begin{align*}
[\,,\,]_\omega: A\ot A \rightarrow A,\quad [x, y]_\omega:= x\bullet_\omega y-y\bullet_\omega x\, \text{ for all }\, x,y,z\in R, \omega\in \Omega.
\end{align*}
Then the pair $(A, \{[\,,\,]_\omega\mid \omega\in \Omega\})$ is a compatible  Lie algebra. \mlabel{them:malie}
\end{theorem}

\begin{proof}
It is sufficient to verify the coupling Jacobi identity.
For $x, y, z \in A$, $\alpha, \beta \in \Omega$, we have
\begin{align*}
[x,[y,z]_\alpha]_\beta=&\ [x, y\bullet_\alpha z-z\bullet_\alpha y ]_\beta\\
=&\ x \bullet_\beta (y\bullet_\alpha z-z\bullet_\alpha y)-(y\bullet_\alpha z-z\bullet_\alpha y) \bullet_\beta x\\
=&\ {x \bullet_\beta (y\bullet_\alpha z)}
-{x \bullet_\beta(z\bullet_\alpha y)}
-{(y\bullet_\alpha z) \bullet_\beta x}
+{(z\bullet_\alpha y) \bullet_\beta x}.
\end{align*}
Also,
\begin{align*}
[y,[z,x]_\alpha]_\beta
=&\ {y \bullet_\beta (z\bullet_\alpha x)}
-{y \bullet_\beta(x\bullet_\alpha z)}
-{(z\bullet_\alpha x) \bullet_\beta y}
+{(x\bullet_\alpha z) \bullet_\beta y},\\
[z,[x,y]_\alpha]_\beta
=&\ {z \bullet_\beta (x\bullet_\alpha y)}
-{z \bullet_\beta(y\bullet_\alpha x)}
-{(x\bullet_\alpha y) \bullet_\beta z}
+{(y\bullet_\alpha x) \bullet_\beta z},\\
[x,[y,z]_\beta]_\alpha
=&\ {x \bullet_\alpha (y\bullet_\beta z)}
-{x \bullet_\alpha(z\bullet_\beta y)}
-{(y\bullet_\beta z) \bullet_\alpha x}
+{(z\bullet_\beta y) \bullet_\alpha x},\\
[y,[z,x]_\beta]_\alpha
=&\ {y \bullet_\alpha (z\bullet_\beta x)}
-{y \bullet_\alpha(x\bullet_\beta z)}
-{(z\bullet_\beta x) \bullet_\alpha y}
+{(x\bullet_\beta z) \bullet_\alpha y},\\
[z,[x,y]_\beta]_\alpha
=&\ {z \bullet_\alpha (x\bullet_\beta y)}
-{z \bullet_\alpha(y\bullet_\beta x)}
-{(x\bullet_\beta y) \bullet_\alpha z}
+{(y\bullet_\beta x) \bullet_\alpha z}.
\end{align*}
Adding all the equations together, we obtain
\begin{align*}
&\ [x,[y,z]_\alpha]_\beta+[y,[z,x]_\alpha]_\beta+[z,[x,y]_\alpha]_\beta+
[x,[y,z]_\beta]_\alpha+[y,[z,x]_\beta]_\alpha+[z,[x,y]_\beta]_\alpha\\
=&\ {x \bullet_\alpha (y\bullet_\beta z)}
+{x \bullet_\beta (y\bullet_\alpha z)}
-{(x\bullet_\beta y) \bullet_\alpha z}
-{(x\bullet_\alpha y) \bullet_\beta z}\\
&\ +{y \bullet_\alpha (z\bullet_\beta x)}
+{y \bullet_\beta (z\bullet_\alpha x)}
-{(y\bullet_\beta z) \bullet_\alpha x}
-{(y\bullet_\alpha z) \bullet_\beta x}\\
&\ +{z \bullet_\alpha (x\bullet_\beta y)}
+{z \bullet_\beta (x\bullet_\alpha y)}
-{(z\bullet_\beta x) \bullet_\alpha y}
-{(z\bullet_\alpha x) \bullet_\beta y}\\
&\ +{(y\bullet_\beta x) \bullet_\alpha z}
+{(y\bullet_\alpha x) \bullet_\beta z}
-{y \bullet_\beta(x\bullet_\alpha z)}
-{y \bullet_\alpha(x\bullet_\beta z)}\\
&\ +{(x\bullet_\beta z) \bullet_\alpha y}
+{(x\bullet_\alpha z) \bullet_\beta y}
-{x \bullet_\beta(z\bullet_\alpha y)}
-{x \bullet_\alpha(z\bullet_\beta y)}\\
&\ +{(z\bullet_\beta y) \bullet_\alpha x}
+{(z\bullet_\alpha y) \bullet_\beta x}
-{z \bullet_\beta(y\bullet_\alpha x)}
-{z \bullet_\alpha(y\bullet_\beta x)}.
\end{align*}
Now each row on the right hand side vanishes since $(A, \{\bullet_\omega\mid \omega\in \Omega\})$  is a  compatible associative algebra. This gives what we need.
\end{proof}

\noindent {\bf Acknowledgments}:
This research is supported by the National Natural Science Foundation of China (Grant No.\@ 11771190, 11771191).

\end{document}